\documentclass[12pt,reqno]{amsart}
\usepackage{mathtools}
\usepackage{amssymb, mathabx, enumitem,amsmath,mathrsfs}
\usepackage{color}
\usepackage[colorlinks=true,linkcolor=blue,urlcolor=blue,citecolor=blue, pagebackref]{hyperref}
\usepackage[alphabetic]{amsrefs}
\usepackage{tikz-cd}
\usepackage{tikz}
\usepackage{tikz-3dplot}
\usetikzlibrary{arrows,cd,snakes}
\usepackage{graphics}
\usepackage{arydshln}

\usepackage[margin=1in]{geometry}

\newcommand{\C}{\mathbb{C}}
\newcommand{\Q}{\mathbb{Q}}
\newcommand{\Z}{\mathbb{Z}}

\newcommand{\op}{\operatorname}

\newcommand{\mf[1]}{\mathfrak{#1}}

\newcommand{\Fl}{\mathrm{Fl}}
\newcommand{\semi}{\frac{\infty}{2}}
\newcommand{\dom}{P^+}

\newcommand{\tkm}{{T_{KM}}}
\newcommand{\tfin}{T}
\newcommand{\ttilde}{{\tilde T}}

\newcommand{\bt}[1]{{\bf t}_{#1}}
\newcommand{\gt}[1]{t^{#1}}

\newtheorem{theorem}{Theorem}[section] 

\newtheorem{remark}[theorem]{Remark}

\newtheorem{corollary}[theorem]{Corollary}

\newtheorem{proposition}[theorem]{Proposition}

\newtheorem{lemma}[theorem]{Lemma}

\newtheorem{definition}[theorem]{Definition}
\newtheorem{definition/lemma}[theorem]{Definition/Lemma}

\newtheorem{example}[theorem]{Example}
\newtheorem{task}[theorem]{Task}

\title{Affine Demazure weight polytopes and twisted Bruhat orders}
\author{Marc Besson}
\address{BICMR, Peking University, Beijing, P.R. China, 100871}
\email{marc@bicmr.pku.edu.cn}
\author{Sam Jeralds}
\address{University of Sydney, Camperdown, NSW 2006}
\email{samuel.jeralds@sydney.edu.au}
\author{Joshua Kiers}
\email{joshkiers@gmail.com}

\begin{document}
\maketitle
\begin{abstract}
For an untwisted affine Kac--Moody Lie algebra $\mf[g]$ with Cartan and Borel subalgebras $\mf[h] \subset \mf[b] \subset \mf[g]$, affine Demazure modules are certain $U(\mf[b])$-submodules of the irreducible highest-weight representations of $\mf[g]$. We introduce here the associated affine Demazure weight polytopes, given by the convex hull of the $\mf[h]$-weights of such a module. Using methods of geometric invariant theory, we determine inequalities which define these polytopes; these inequalities come in three distinct flavors, specified by the standard, opposite, or semi-infinite Bruhat orders. We also give a combinatorial characterization of the vertices of these polytopes lying on an arbitrary face, utilizing the more general class of twisted Bruhat orders.  

\end{abstract}

\section{Introduction}
\subsection{Motivation} 
Let $G$ be a finite-dimensional, simple, simply-connected complex algebraic group, with a fixed choice of maximal torus and Borel subgroup $T \subset B \subset G$. Let $\lambda$ be a dominant integral weight of $T$, let $V_{\lambda}$ be the associated irreducible, finite-dimensional representation of $G$, and let $w \in  W$ be an element of the Weyl group. We associate to such a $\lambda$ and $w$ the Demazure module $V_\lambda^w$, which is the smallest cyclic $B$-submodule of $V_\lambda$ containing the extremal weight vector $v_{w\lambda}$. These modules and their $T$-characters were first studied by Demazure \cite{Dem1, Dem2}, particularly via their geometric constructions as global sections of line bundles on Schubert varieties $X_w = \overline{BwB/B} \subset G/B$ in the flag variety associated to $G$. 

To each Demazure module $V_\lambda^w$ one can attach the associated \emph{Demazure weight polytope}
\begin{equation}  \label{polydefintro}
P_\lambda^w:= \op{conv}_\Q \{ \mu: V_\lambda^w(\mu) \neq 0\} \subset X^\ast(T) \otimes \Q,
\end{equation}
where $V_\lambda^w(\mu)$ denotes the $\mu$ weight space of $V_\lambda^w$ under the action of $T$, $X^\ast(T)$ is the character lattice of $T$, and $\op{conv}_\Q$ is the rational convex hull. Such weight polytopes, their geometric and combinatorial descriptions, and their connection to the characters of Demazure modules have been of recent interest. When $G$ is of type $A_n$, the role of a Demazure character is played by a key polynomial $\kappa_\alpha$, and the Demazure weight polytope is given by the corresponding Newton polytope $\op{Newton}(\kappa_\alpha)$. A conjecture of Monical--Tokcan--Yong \cite{MTY}, proven by Fink--M{\'e}sz{\'a}ros--St. Dizier \cite{FMSD}, gives a relation between the lattice points in $\op{Newton}(\kappa_\alpha)$ and the monomials appearing in $\kappa_\alpha$; that is, the key polynomial $\kappa_\alpha$ has \textit{saturated Newton polytope} (see \cite{MTY}*{Definition 1.1}). The polytopes $\op{Newton}(\kappa_\alpha)$ can also be viewed as certain examples of Bruhat-interval polytopes, introduced by Tsukerman and Williams \cite{TW}; these more general polytopes--and their faces--have concise descriptions both by their vertices and by their inequalities.

In previous work \cite{BJK}, the authors extended the consideration of Demazure weight polytopes $P_\lambda^w$ beyond type $A_n$. First, using geometric invariant theory (GIT) techniques, inequalities were derived which cut out $P_\lambda^w$ in all types. More specifically, for points $x$ in the Schubert variety $X_w$, a line bundle $\mathbb{L}$, and a cocharacter $\eta$ of $T$, one gets inequalities $\mu^{\mathbb{L}}(x,\eta)\le 0 $ from the Hilbert--Mumford criterion for semistability. The set of inequalities was then reduced to a necessary and sufficient set by translating the geometric data coming from the GIT problem into the combinatorics of the Demazure product or monoid structure $(W, \ast)$ on the Weyl group. These were then used to describe the combinatorial structure on the faces of $P_\lambda^w$, which made apparent that they are again Demazure weight polytopes. Finally, with this in hand, representation-theoretic techniques could be applied to extend the saturation result of \cite{FMSD} to Demazure characters when $G$ is simple of classical type. By additional computational methods, the same result was also obtained in types $F_4$ and $G_2$, and is conjectured to hold for $E_6, E_7 $ and $E_8$ (see \cite{BJK}*{Conjecture 11.4}).

\subsection{Extending to the affine setting} In the present work, we aim to extend a portion of this story to the setting of \emph{affine Demazure modules}.  We do not endeavor to give a complete survey of these modules, but refer to Section \ref{affDem} for a brief introduction to the necessary algebraic and geometric perspectives. Like their ``finite-type" (non-affine) counterparts, the affine Demazure modules are determined by a dominant integral weight $\lambda$ and Weyl group element $w$, but these are now for the \emph{untwisted affine Kac-Moody groups}, and the \emph{affine Weyl group}. 

Via the same definition as in (\ref{polydefintro}), we can define, for an affine Demazure module $V_\lambda^w$, the corresponding affine Demazure weight polytope $P_\lambda^w$. The goal of this paper is two-fold:

\begin{enumerate}
\item Explicitly describe a complete (though redundant) set of inequalities defining an affine Demazure polytope $P_\lambda^w$ (Theorem \ref{first-pass}), and
\item Give a combinatorial description of the vertices of $P_\lambda^w$ which lie on an arbitrary face (Theorem \ref{face-intersection-2}).
\end{enumerate}

For the first objective, we again use GIT techniques, now on Schubert varieties in the affine flag variety $G(\mathcal{K})/I$, where $G$ is as before, $\mathcal{K} = \C((t))$ is the field of Laurent series, and $I$ is the (positive) Iwahori subgroup associated to $B$. While the approach is the same as in the finite case, we encounter new phenomena that are unique to the affine setting, stemming from how an arbitrary cocharacter relates to the Tits cone. This new feature manifests itself as follows: the inequalities defining $P_\lambda^w$ (and hence the faces of $P_\lambda^w$) come in three distinct types, depending on a cocharacter $\eta$ being 
\begin{itemize}
\item[(i)] in the dominant chamber, or 
\item[(ii)] in the anti-dominant chamber, or
\item[(iii)] outside of the Tits cone and dominant on the subspace $\op{Lie}(T)$;
\end{itemize}
in more familiar terms of weights, these three cases correspond to being of positive level, negative level, or level zero. These three types lead to  the associated geometry in $G(\mathcal{K})/I$ being ``thin," ``thick," or ``semi-infinite" in flavor and to the associated combinatorics being with respect to the standard, opposite, or semi-infinite Bruhat orders. While these three orders are considerably different, a crucial shared property is that each satisfies the \emph{diamond lemma} (recalled in Lemma \ref{diamond}). This allows us to give a unified treatment of these orders and define an associated Demazure product for each in Section \ref{Orders and Demazure products}. Note that the inequalities obtained in Theorem \ref{first-pass} are in general not an irredundant set of inequalities. It would be an interesting problem to determine the minimal set of inequalities. 

For the second objective, even the introduction of the three distinct Bruhat orders does not successfully capture the combinatorics of the faces of $P_\lambda^w$. We instead rely on the more general notion of \emph{twisted Bruhat orders}, as introduced by Dyer \cite{Dyer3, Dyer4}, which remember more subtle information regarding the coweight $\eta$ beyond its ``type" as above. As a common generalization of the standard, opposite, and semi-infinite Bruhat orders, twisted Bruhat orders share the same beneficial properties such as the diamond lemma and well-defined Demazure products. We recall the definition of these orders, and discuss the key features which generalize familiar results on the standard Bruhat order, in Section \ref{orderable-horrors}. While an interesting topic, we do not discuss in this paper the associated geometric characterizations of twisted Bruhat orders as we do for the standard, opposite, and semi-infinite cases; we refer interested readers instead to the work of Billig and Dyer \cite{BD}. We do not give any indication here about the saturation of affine Demazure characters, although this remains a question of interest. 

\subsection{Outline of paper}
In Section \ref{affDem} we recall the definitions of the affine (Kac--Moody) Lie algebras and groups, their irreducible highest weight modules, and the Demazure modules $V_\lambda^w$. We also give a sketch of the geometric construction of the Demazure modules via the Borel--Weil--Bott theorem and line bundles on the affine flag variety and Schubert varieties. 

In Section \ref{Orders and Demazure products} we discuss the standard, opposite, and semi-infinite Bruhat orders. We give a uniform combinatorial treatment of these orders and their associated length functions, and in particular use the diamond lemma to define a Demazure product for each of these orders. We show that these products are well-defined and derive the combinatorial properties of this product that we will need. 

In Section \ref{geom-orders}, we recast the three Bruhat orders into geometric statements via closure relations in the affine flag variety. For the standard and opposite Bruhat orders these are well-known; for the semi-infinite order, this is also well-known, although the existing literature suffers from multiple common conventions. We fix our conventions with some explicit explanation in this setting. 

In Section \ref{GITsection}, we briefly recall how to use GIT techniques to produce controlling inequalities for semistability. We then apply this to line bundles on Schubert varieties in Section \ref{ineqdetermined} to determine the inequalities of $P_\lambda^w$. 

In Section \ref{Stabilizer section}, we examine the stabilizer of a coweight $\eta$ as a subgroup in the affine Weyl group. We introduce a uniform notation for distinguished subgroups $W(\eta)$ related to the stabilizer and their coset representatives $W^{(\eta)}$. In Section \ref{facesandtwisteds} we examine the faces of the polytopes $P_\lambda^w$ to motivate in Section \ref{orderable-horrors} the introduction of the twisted Bruhat orders of Dyer. We derive the key combinatorial properties of these orders and describe how they relate to the cosets $W^{(\eta)}$ and subgroups $W(\eta)$ of the previous sections. Finally, in Section \ref{hummingbirds} we use this machinery to determine which vertices of $P_\lambda^w$ lie on a fixed face, and show that this set is itself determined by an interval in the Coxeter group $W(\eta)$ beginning at the identity.

\subsection*{Acknowledgements} The authors would like to thank Jiuzu Hong for encouraging the pursuit of this question in the affine Lie algebras setting, and for many helpful discussions on the geometry of affine flag varieties.

\section{Preliminaries on affine Lie algebras and affine flag varieties} \label{affDem}
In this section, we briefly sketch the construction of affine (Kac--Moody) Lie algebras, their irreducible highest-weight representations, and the associated affine flag varieties.  Because of the scope of work in these areas (in particular with relation to geometry), many differing but equivalent choices of conventions and realizations exist in the literature. We fix here our notations and conventions, following most closely those of \cite{Kac} and \cite{KumarBook}, with the translation to other existing conventions being a straightforward exercise for the knowledgeable reader.

\subsection{Affine Lie algebras}\label{affine-lie-algebras}

Let $\mathring{\mf[g]}$ be a finite-dimensional simple Lie algebra over $\C$. We fix a choice of Cartan and Borel subalgebras $\mathring{\mf[h]}$ and $\mathring{\mf[b]}$ with $\mathring{\mf[h]} \subset \mathring{\mf[b]} \subset \mathring{\mf[g]}$. With respect to this choice, we have a root system $\mathring{\Phi}$ with simple roots $\{\alpha_1, \dots, \alpha_n\}$. We also set $\theta \in \mathring{\Phi}$ to be the highest root of $\mathring{\mf[g]}$. Then there is a normalized invariant form $( \cdot | \cdot )$ on $\mathring{\mf[g]}$ satisfying $(\theta | \theta)=2$; we will always use this normalization. The associated root lattice of $\mathring{\mf[g]}$ is given by $\mathring{Q}:= \bigoplus_{i=1}^n \Z \alpha_i \subset \mathring{\mf[h]}^\ast$, where we identify $\mathring{\mf[h]}^\ast$ with $\mathring{\mf[h]}$ via $(\cdot | \cdot)$. We have similarly the coroot lattice $\mathring{Q}^\vee = \bigoplus_{i=1}^n \Z \alpha_i^\vee \subset \mathring{\mf[h]}$. The Weyl group of $\mathring{\mf[g]}$, denoted by $\mathring{W}$, is generated as a Coxeter group by the simple reflections $\{ s_1, \dots, s_n \}$, where $s_i$ is the reflection corresponding to the simple root $\alpha_i$. 

To this data, we let $\mf[g]$ be the untwisted affine Kac--Moody Lie algebra associated with $\mathring{\mf[g]}$; we refer to these throughout simply as \emph{affine Lie algebras}. As a vector space, $\mf[g]$ is given by 
$$
\mf[g] = \left(\mathring{\mf[g]} \otimes_\C \C[t, t^{-1}]\right) \oplus \C d \oplus \C K,
$$
where $\C[t,t^{-1}]$ is the ring of Laurent polynomials, $d$ is the degree derivation, and $K$ is the central element. Then $\mathring{\mf[g]}$ is naturally a subalgebra of $\mf[g]$ via the embedding $x \mapsto x \otimes 1$. We let 
$$
\mf[h] = \mathring{\mf[h]} \oplus \left(\C d + \C K\right)
$$ 
be the Cartan subalgebra of $\mf[g]$. We extend the bilinear form $( \cdot | \cdot )$ to $\mf[h]$ by $(\mathring{\mf[h]} | \C d \oplus \C K)=(d | d) = (K | K) =0$ and $(d | K)=1$. Similarly, we can define
$$
\mf[h]^\ast = \mathring{\mf[h]}^\ast \oplus \left(\C \Lambda_0 + \C \delta\right),
$$
where $\Lambda_0$ is dual to $K$ and $\delta$ is dual to $d$ (that is, $\langle \Lambda_0, K\rangle =\langle \delta, d\rangle=1$ and $\langle\Lambda_0,\mathring{\mf[h]} \oplus \C d\rangle = \langle \delta, \mathring{\mf[h]} \oplus \C K\rangle=0$). 

As a Kac--Moody algebra, $\mf[g]$ has a root system $\Phi$ which decomposes as a disjoint union $\Phi_{re} \sqcup \Phi_{im}$ of real and imaginary roots, respectively. These are given by 
$$
\begin{aligned}
\Phi_{re}&=\{\beta +n\delta: \beta \in \mathring{\Phi}, \ n \in \Z\}, \\
\Phi_{im}&=\{n \delta: \ n \in \Z, n \neq 0 \}.
\end{aligned}
$$
The simple roots of $\mf[g]$ are given by $\{\alpha_0, \alpha_1, \dots, \alpha_n\}$, where as before $\{\alpha_1, \dots, \alpha_n\}$ are the simple roots in $\mathring{\Phi}$ and $\alpha_0:=-\theta+\delta$. Then the \emph{affine Weyl group} $W$ is generated as a Coxeter group by the simple reflections $\{s_0, s_1, \dots, s_n \}$. The affine Weyl group $W$ can also be realized as an affine reflection group associated to the finite Weyl group $\mathring{W}$ via 
$$
W \cong \mathring{W} \ltimes \mathring{Q}^\vee,
$$
where to an element $\beta^\vee \in \mathring{Q}^\vee$ of the finite coroot lattice we associate the translation $\bt{\beta^\vee}$. In this presentation, we have $s_0=s_\theta \bt{-\theta^\vee}$.

\subsection{Demazure modules for affine Lie algebras}

Given an affine Lie algebra $\mf[g]$, recall the set of integral dominant weights for $\mf[g]$ (or more appropriately, for $\mf[h] \subset \mf[g]$) defined by 
$$
\dom := \{\lambda \in \mf[h]^\ast: \ \langle\lambda, \alpha_i^\vee\rangle \in \Z_{\geq 0} \ \forall i = 0, 1, \dots, n\}.
$$
We denote by $\Lambda_i \in \dom$ the $i^{th}$ fundamental weight of $\mf[g]$, defined uniquely by $\langle\Lambda_i ,\alpha_j^\vee\rangle=\delta_{ij}$, the Kronecker delta, and $\langle\Lambda_i,d\rangle=0$ for all $i=0,1,\dots,n$. Note that the fundamental weight $\Lambda_0$ agrees with the previous definition. We define the weight lattice of $\mf[g]$, $P$, via 
$$
P:= \bigoplus_{i=0}^n \Z \Lambda_i \oplus \C \delta.
$$

To each integral dominant weight $\lambda \in \dom$, we associate the corresponding integrable irreducible highest-weight $\lambda$ representation of $\mf[g]$, $V_\lambda$. This gives a bijection between $\dom$ and the collection of such representations. For the purposes of this paper, we will always assume $\langle \lambda,K\rangle >0$, as the case $\lambda = k\delta$ is uninteresting ($\dim V_{k\delta}=1$). As an $\mf[h]$-module via restriction, $V_\lambda$ has a weight space decomposition 
$$
V_\lambda=\bigoplus_{\mu \in P} V_\lambda(\mu),
$$
where $V_\lambda(\mu):=\{v \in V_\lambda: \ h.v=\langle \mu,h\rangle v \ \forall h \in \mf[h]\}$; note that each weight space is finite dimensional. The set of weights of $V_\lambda$ is given by $\{\mu \in P: V_\lambda(\mu) \neq 0\}$. This is a (typically infinite) $W$-invariant set of weights in $P$. If $\mu=w\lambda$ is a weight of $V_\lambda$ for some Weyl group element $w$, we say that $\mu$ is an extremal weight of $V_\lambda$. 

As $V_\lambda$ is a highest-weight representation, there exists a unique-up-to-scaling nonzero vector $v_\lambda \in V_\lambda(\lambda)$, the highest-weight vector. For any Weyl group element $w \in W$, consider the extremal weight space $V_\lambda(w\lambda)$. This is one-dimensional, with basis element $v_{w\lambda}$. We now introduce the key representation-theoretic object of interest; namely, the Demazure module $V_\lambda^w$. 

\begin{definition}
Let $V_\lambda$ be an irreducible, highest-weight $\mf[g]$-representation, and fix $w \in W$. Then the \emph{Demazure module} $V_\lambda^w$ is the $U(\mf[b])$-submodule of $V_\lambda$ given by 
$$
V_\lambda^w:=U(\mf[b]).v_{w\lambda} \subset V_\lambda.
$$
\end{definition}
\noindent In the next subsection, we give a geometric construction for $V_\lambda^w$. 

\begin{remark}
Many authors consider, for affine Lie algebras, Demazure modules $V_\lambda^w$ for $w$ an element of the \emph{extended} affine Weyl group. These modules can equivalently be described by an element of the usual affine Weyl group along with a suitable adjustment to the highest weight, so we do not make a distinction in these cases. 
\end{remark}

We can again consider the $\mf[h]$-module structure via restriction for $V_\lambda^w$; this gives a similar weight space decomposition 
$$
V_\lambda^w = \bigoplus_{\mu \in P} V_\lambda^w(\mu).
$$
A basic question about such representations is to understand the set of nonzero weights of $V_\lambda^w$, denoted $\op{wt}(V_\lambda^w)$. This set is calculable via the Demazure character formula through the work of Kumar \cite{Kumar2} and Mathieu \cite{Mat}. In this work we consider the following a priori weaker object, the \emph{Demazure weight polytope} $P_\lambda^w$. These are finite-dimensional, compact, convex, lattice polytopes and will be the fundamental objects of interest for us.

\begin{definition}
Let $V_\lambda^w$ be a Demazure module associated to an affine Lie algebra $\mf[g]$. Then the Demazure weight polytope $P_\lambda^w$ is defined by 
$$
P_\lambda^w:=\op{conv}(\op{wt}(V_\lambda^w)) \subset \mf[h]^\ast_{\Q}.
$$
\end{definition}
While not clear from the definition, the Demazure weight polytope $P_\lambda^w$ has a simple description via its vertices. We take the following result from \cite{BJK}*{Theorem 4.7}, as the proof given there goes through to the affine Lie algebra setting; in the finite case, this result is implicit in the work of Dabrowski \cite{Dab}*{page 119} via a similar proof. 

\begin{proposition} \label{convhull}
For any $\lambda \in \dom$ and $w \in W$, we have 
$$
P_\lambda^w = \op{conv}(\{ v \lambda: v \leq w \})
$$
where $v \leq w$ is in the standard Bruhat order on $W$. 
\end{proposition}

\subsection{Affine flag varieties} 

Our ultimate goal is to understand the inequalities defining the polytopes $P_\lambda^w$, using techniques from geometric invariant theory (GIT) in 
Section \ref{ineqdetermined}. Toward that end, we first need to understand the connection between Demazure modules and the geometry of Schubert varieties in the affine flag variety. 

To simplify the notation for the remainder of the paper, let $G$ be the finite-dimensional, simple, simply-connected complex algebraic group with $\op{Lie}(G) \cong \mathring{\mf[g]}$.  Associated to $G$ we have group schemes $LG$ and $L^+G$ such that $LG(R)=G(R((t)))$ and $L^+G(R)=G(R[[t]])$ for any $\mathbb{C}$-algebra $R$. We write also $L^-G(R)=G(R[t^{-1}])$. $L^+G$ and $L^{-}G$ are equipped with evaluation morphisms $ev_0: L^+G \rightarrow G$ and $ev_{-\infty} : L^-G \rightarrow G$ sending the uniformizers $t$ and $t^{-}$ to $0$, respectively. We define $I:= ev_0^{-1}(B)$ and $I^{-} := ev_{\infty}^{-1}(B^-)$. Note that the root spaces appearing in $I$ and $I^{-1}$ partition the affine roots associated to $LG$. 

We may form the affine flag variety as the \'etale sheafification of  \[R \mapsto LG(R)/I(R).\] This is an ind-projective ind-scheme (for $G$ as considered). As we will just concern ourselves with the $\mathbb{C}$-points, we have that $Fl_G = G(\mathcal{K})/I$. Our assumptions on $G$ (simple, simply-connected) imply that $Fl_G$ is connected: indeed for $G$ reductive, we have $\pi_0(Fl_G)=\pi_1(G)=P^{\vee}/Q^{\vee}$.

Using the 
Bruhat decomposition of $LG$, given by 
$$
\begin{aligned}
 LG &= \bigcup_{w \in W} IwI,
 \end{aligned}
 $$
 we obtain a stratification of $Fl_G$ by finite-dimensional strata. We write $C(w)$ for $IwI/I$ and $\mathbf{X}(w)$ for its closure, $\overline{IwI/I}$. 
 In general these are singular varieties which are, however, normal and Cohen-Macaulay and admit rational resolutions of singularities. These Schubert varieties $\mathbf{X}(w)$, as well as their cousins in other flag varieties (whose definition will be recalled later), are the main geometric objects of this paper.

\subsection{Line bundles}

Let $\hat{G}$ be the affine Kac-Moody group associated to $G$, a simple group over $\mathbb{C}$. It is built via two successive extensions:

\[1 \rightarrow \mathbb{G}_m\langle K \rangle \rightarrow \tilde{G} \rightarrow LG \rightarrow 1\] and \[1 \rightarrow \mathbb{G}_m\langle d \rangle \rightarrow \hat{G} \rightarrow \tilde{G} \rightarrow 1.\] The extensions split over $L^+G$. As is well-known, $Fl_G$ is a homogeneous space under $LG$, but the line bundles on $Fl_G$ cannot be given an $LG$-equivariant structure. This is remedied by the fact that line bundles on $Fl_G$ may be given a $\hat{G}$-linearization. These linearizations may be restricted to $\tkm$, the Kac--Moody torus, which fits into the following diagrams:

\begin{center}
\begin{tikzcd}
\ttilde \arrow[r] \arrow[d] & \tfin \arrow[d] \\ \tilde{G} \arrow[r] & LG \\
\end{tikzcd}
\end{center}

\begin{center}
	\begin{tikzcd}
	\tkm \arrow[r] \arrow[d] & \ttilde \arrow[d] \\ \hat{G} \arrow[r] & \tilde{G} \\
	\end{tikzcd}
\end{center}

Note that $\tkm$ is an algebraic torus, finite-dimensional over $\mathbb{C}$, of dimension $rk(G)+2$. As mentioned before, since the extensions split over $L^+G$, $\tkm$ is a split extension $\tfin \times \mathbb{G}_m\langle K \rangle \times \mathbb{G}_m \langle d \rangle$.

For any integral $\lambda \in P$ we may construct an equivariant line bundle $L(\lambda)$ on $\Fl_G$: in particular in our setting we have \[Pic(\Fl_G) \simeq P.\] We follow the normalization that for $\lambda \in P^+$, $L(\lambda)$ will have global sections; or in other words the duals $L(-\lambda)= L(\lambda)^*$ can be identified with the pullback of a tautological bundle $L(-\lambda)= i^*\mathscr{O}_{\mathbb{P}(V_{\lambda})}(-1)$ associated to a Plucker embedding $i:\Fl_G \rightarrow \mathbb{P}(V_{\lambda})$ (cf. \cite{KumarBook}*{7.2.1}).

These line bundles may be restricted to line bundles on $\mathbf{X}(w)$, 
and they will be denoted by $L_{w}(\lambda)$. 

\begin{remark}
	Though not strictly necessary for this paper, the following 
theorem of Mathieu \cite{Mat}*{Chapitre XII Prop. 6} is relevant: all line bundles on $\mathbf{X}(w)$ arise in this fashion as restrictions of line bundles from $\Fl_G$. More precisely, given $w \in W$ and writing $Supp(w)$ for its support, we can define \[P^{\circ}_w= \{\lambda \in P| \langle \lambda, \alpha_i^{\vee} \rangle= 0~~ \forall i \in Supp(w)\}.\] Then $Pic(\mathbf{X}(w)) \simeq P/P^{\circ}_w$.
\end{remark}

The link between the algebraic definition of Demazure modules $V_{\lambda}^w$ comes from the following well-known result \cite{KumarBook}:

\begin{proposition}\label{affine-borel-weil} For $\lambda \in P^+$,
	$V_{\lambda}^w \cong H^0(\mathbf{X}(w), L_w(\lambda))^*$, and all higher cohomology groups vanish.
\end{proposition}

\noindent This realization of the Demazure module $V_\lambda^w$, and the fact that the Schubert variety ${\bf X}(w)$ is $\tkm$-stable, situates us nicely in the realm of GIT techniques, as we will see in Section \ref{GITsection}.

\section{Orders and Demazure products on the affine Weyl group} \label{Orders and Demazure products}

In this section, we build the combinatorial framework which will support and reinterpret the subsequent geometric properties we are interested in. In particular, we closely examine an affine Weyl group $W$ and various ``Bruhat" partial orders on it; to each of these, we then introduce an associated ``Demazure product," defined purely combinatorially, and investigate their shared properties. While the individual orders and Demazure products greatly vary, a combinatorial lifting lemma (Lemma \ref{diamond}) allows us to treat these products in a uniform fashion. 

The orders considered in this section and their associated Demazure products are all common examples of a more general construction for \emph{twisted Bruhat orders}, which we will recall in Section \ref{orderable-horrors}. However, we separate and highlight in this section the most familiar of these orders--namely, the standard, opposite, and semi-infinite Bruhat orders--as a first introduction to this approach, as these orders will play the most significant roles in our geometrical analysis in subsequent sections.

\subsection{Three Bruhat orders on affine Weyl groups}

Let $W$ be an affine Weyl group, with simple generators $S=\{s_0, s_1, \dots, s_n\}$ as in Section \ref{affine-lie-algebras}. As $(W,S)$ is a Coxeter group, we have an associated length function $l: W \to \Z_{\geq 0}$ and Bruhat order $\leq$ on $W$; we will refer to this as the \textit{standard} or \textit{positive} Bruhat order on $W$. As the Bruhat order is a partial order on W graded by the length function $l$, we will often consider this data as a pair $(l, \leq)$. We assume familiarity with the definition and standard properties of the Bruhat order, and refer to \cite{BjBr} as our primary source for the full details and constructions.

Second, to any Coxeter group $(W,S)$, we can associate the \textit{opposite} or \textit{negative} Bruhat order as a slight alteration of the standard Bruhat order. 

\begin{definition} The opposite, or negative, Bruhat order $\leq_-$ on $(W,S)$ is defined by 
$$
w \leq_- v \iff v \leq w;
$$
that is, $\leq_-$ is the reverse order of $\leq$ on $W$. Then $\leq_-$ is a graded partial order on $W$ graded by the opposite, or negative, length function 
$$
l_-: W \to \Z_{\leq 0}, \ l_-(w):=-l(w).
$$
\end{definition}

As $(l_-, \leq_-)$ is just the reverse order of the standard Bruhat order with the negative standard length function, many of the combinatorial properties of $(l_-, \leq_-)$ follow immediately as rephrasings from those of $(l, \leq)$. Thus, proofs of properties for $(l_-, \leq_-)$ typically reduce to the case of $(l, \leq)$ in the natural way. 

While the standard and opposite Bruhat orders make sense for any Coxeter group, we now recall a third order on $W$, the \textit{semi-infinite Bruhat order}, specific to the affine Weyl group setting. This relies on the semi-direct product construction of $W$. In particular, recall that we have the realization 
$$
W \cong \mathring{W} \ltimes \mathring{Q}^\vee,
$$
where $\mathring{W}:=\langle s_1, \dots, s_n \rangle$ is the finite Weyl group, and $\mathring{Q}^\vee$ is the finite coroot lattice acting via translations $\bt \xi$ for $\xi \in \mathring{Q}^\vee$. As before, we have in this realization
$$
s_0 = s_\theta \bt{-\theta^\vee},
$$
where $\theta$ is the highest root of the associated finite-type root system $\mathring{\Phi}$. Let $\rho$ be the half sum of the positive roots of $\mathring{\Phi}$, the (finite) Weyl vector.

\begin{definition}
The semi-infinite length function $l_\semi$ is defined by 
$$
l_\semi: W \to \Z, \ l_\semi(w\bt \xi):= l(w) + \langle 2\rho, \xi \rangle,
$$
where $l(w)$ is the standard length of $w \in \mathring{W}$.
\end{definition}

Using this length function, we can define a partial order on $W$ in a fashion analogous to $(l, \leq)$ or $(l_-, \leq_-)$ as follows. Let $\beta \in \Phi^+_{re}$ be a positive real affine root and $w \in W$. We denote by 
$$
w \xrightarrow {\beta} s_\beta w
$$
if $l_\semi(s_\beta w)=l_\semi(w)+1$, where $s_\beta \in W$ is the reflection associated to $\beta$. We can then define the semi-infinite Bruhat order. 

\begin{definition} For $w, v \in W$, we say that $w \leq_\semi v$ if there exists a (possibly empty) sequence of positive real affine roots $\beta_1, \dots, \beta_k$ such that 
$$
w \xrightarrow{\beta_1} s_{\beta_1}w \xrightarrow{\beta_2} s_{\beta_2}s_{\beta_1}w \xrightarrow{\beta_3} \cdots \xrightarrow{\beta_k}s_{\beta_k}\cdots s_{\beta_1}w=v.
$$
Then $\leq_\semi$ is a partial order on $W$ graded by $l_\semi$.
\end{definition}

\begin{remark}
Note that $l_\semi(v)$ and $l_\semi(v^{-1})$ are not necessarily the same. In fact, we have chosen a ``left-sided'' convention for the covering relations: $w \xrightarrow{\beta} s_\beta w$. The ``right-sided'' semiinfinite Bruhat order is equivalent but less suitable for the present context. 
\end{remark}

The semi-infinite Bruhat order $(l_\semi, \leq_\semi)$ appears in the literature in relation to Lusztig's \textit{generic order} \cite{Lus} and Peterson's \textit{stable order} \cite{Pet}. Like the more familiar Bruhat orders, the semi-infinite order has found extensive application in the representation theory of affine Lie algebras and the geometry of flag varieties; see for example \cite{Ish}*{Introduction} and citations therein. The semi-infinite Bruhat order has a rich combinatorial structure, of which we introduce only a portion for our purposes.  In Section \ref{geom-orders}, we will recall the geometric description of these orders in the affine flag variety.

\begin{example} \label{intervalex} 

In type $A_2^{(1)}$, here is a portion of the Hasse diagram for the standard order local to $s_2s_1s_0 = s_1 {\bf t}_{-\alpha_1^\vee-\alpha_2^\vee}$. This is not a full interval, but merely a sampling of elements and how they are related. The grading along the top indicates the standard length, on the bottom is the negative length. The covering relations for the opposite order would have all arrows exactly reversed. 

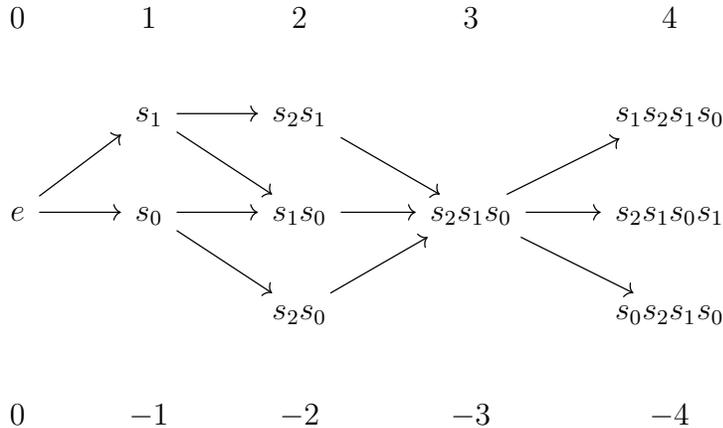
\begin{center}
\begin{tikzcd}
& 0 & 1 & 2 &     3        & 4 & & \\ 
& & s_1 \arrow[r] \arrow[dr]  & s_2s_1 \arrow[dr]  &             & s_1s_2s_1s_0 & & \\ 
 & e  \arrow[r] \arrow[ur]& s_0 \arrow[r] \arrow[dr]  & s_1s_0 \arrow[r] & s_2s_1s_0 \arrow[r]  \arrow[ur] \arrow[dr] & s_2s_1s_0s_1 & & \\ 
 & & & s_2s_0 \arrow[ur] &             & s_0s_2s_1s_0 & &  \\ 
 & 0 & -1 &-2  &      -3       & -4 & & 
\end{tikzcd}
\end{center}

Now here is a portion of the Hasse diagram for the semi-infinite order, same elements (again, not a full interval): 

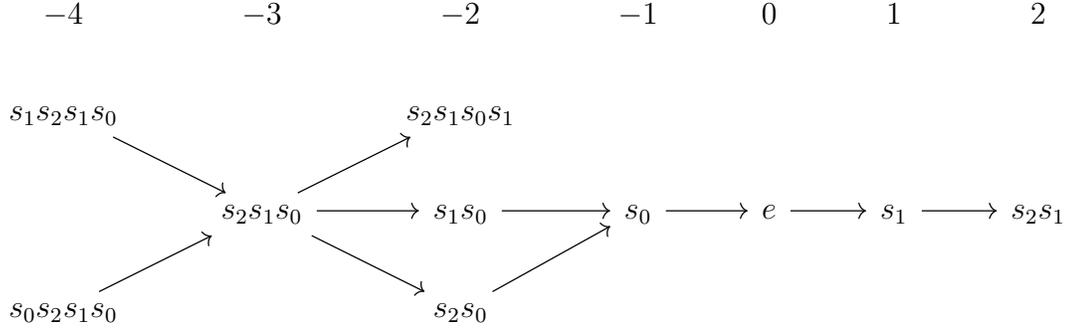
\begin{center}
\begin{tikzcd}
& -4 &    -3       & -2 & -1 & 0 & 1 & 2  \\ 
 & s_1s_2s_1s_0 \arrow[dr] &           & s_2s_1s_0s_1 & & & &  \\ 
 & &     s_2s_1s_0   \arrow[ur]  \arrow[r] \arrow[dr] &  s_1s_0 \arrow[r] & s_0 \arrow[r] & e \arrow[r] & s_1 \arrow[r] & s_2s_1 \\
 &s_0s_2s_1s_0 \arrow[ur] &           & s_2s_0 \arrow[ur] & & & & 
\end{tikzcd}
\end{center}

\end{example}

As seen in Example \ref{intervalex}, the global structure of intervals in the affine Weyl group with respect to the standard, opposite, and semi-infinite Bruhat orders can be quite varied. However, these orders share the following critical local property known in the literature as the \textit{lifting lemma}, the \textit{diamond lemma}, or the \textit{Z-property}.  For the standard Bruhat order, the diamond lemma as we state it comes from \cite{BGG}*{Lemma 2.5}; the analogous statement for the opposite Bruhat order follows immediately. For the semi-infinite Bruhat order, versions of the diamond lemma are given in \cite{INS}*{Lemma 4.1.4} and \cite{Ish}*{Lemma 2.6}. Our phrasing of the diamond lemma in this setting is a direct consequence of these.

\begin{lemma} \label{diamond}
Let $(l_\diamond, \leq_\diamond)$ be any of the standard, opposite, or semi-infinite Bruhat orders with their respective length functions. Let $w, v \in W$ such that $w <_\diamond v$, and let $s \in S$ be a simple reflection. Then 
\begin{itemize}
\item [(a)] Either $sw \leq_\diamond v$ or $sw <_\diamond sv$. 
\item [(b)] Either $w \leq_\diamond sv$ or $sw <_\diamond sv$. 
\end{itemize}
\end{lemma}

\begin{example} 

Again in type $A_2^{(1)}$, with the semi-infinite order, take for instance $w=s_1s_0$ and $v=s_0$; note that we have $s_1s_0 \leq_\semi s_0$. We get this diamond with $s=s_2$:

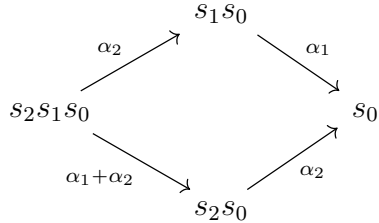
\begin{center}
\begin{tikzcd}
 & s_1s_0 \arrow[dr,"\alpha_1"]& \\ 
 s_2s_1s_0  \arrow[ur,"\alpha_2"]  \arrow[dr,"\alpha_1+\alpha_2"'] & & s_0  \\ 
 & s_2s_0  \arrow[ur,"\alpha_2"'] & 
\end{tikzcd}
\end{center}

\end{example}

The notation $(l_\diamond, \leq_\diamond)$ to denote an arbitrary choice of our three Bruhat orders is intentional. In what follows, many of the combinatorial constructions we introduce can be treated uniformly across the three orders provided that the order and length functions satisfy their respective diamond lemma. Thus we use the subscript $\diamond$ to highlight this property while also treating each case simultaneously.

\subsection{Demazure products associated to $(l_\diamond, \leq_\diamond)$} \label{Demprodsection}

Given a Bruhat order $(l_\diamond, \leq_\diamond)$, we next introduce the notion of the related \textit{Demazure product}. For the standard Bruhat order, this is the usual Demazure product, also called the 0-Hecke monoid structure or Coxeter monoid structure in the literature. First studied combinatorially by Norton \cite{Nor} in the setting of the Hecke algebra, the Demazure product gives an alternative associative monoid structure on any Coxeter group $W$, which we denote by $(W, \ast)$.  In alignment with our notational conventions for the length function and Bruhat order, we do not put any subscript on $\ast$ when discussing the standard Demazure product. 

We generalize this construction to produce two additional ``Demazure products" on $W$. The first, associated to the opposite Bruhat order, was considered by He \cite{He1} and called the \textit{downward} Demazure product. This has been used in the study of the ``wonderful compactification" of an adjoint semisimple group \cite{He1}, \cite{He2} and the Newton points and strata of affine Deligne--Lusztig varieties \cite{Sad}. The second is an analogous product in the semi-infinite setting, which we believe to be new. 

To begin, we record the following technical lemma which will enable us to give a uniform construction of the Demazure products. This is clear in the standard and opposite Bruhat order cases; for a proof in the semi-infinite case, see \cite{INS}*{Lemma 4.1.2}. 

\begin{lemma} \label{lengthpm1}
Let $w \in W$ and let $s \in S$ be an arbitrary simple reflection. Then $$l_\diamond(sw)=l_\diamond(w) \pm 1,$$ so that $w <_\diamond sw$ or $sw <_\diamond w$. 
\end{lemma}

With this in hand, since we now have that for any $w \in W$ and simple reflection $s$ that the pair $\{w, sw\}$ are comparable in the ordering $\leq_\diamond$, we can define the Demazure products in two steps. 

\begin{definition} \label{simpleDemprod}
Let $v \in W$ and $s \in S$ be a simple reflection. Then we define the Demazure product $s \ast_\diamond v$ via 
$$
s \ast_\diamond v := \op{max}_{\leq_\diamond} \{v, sv\} = \begin{cases} v, & l_\diamond(sv)=l_\diamond(v)-1. \\ sv, & l_\diamond (sv)=l_\diamond (v)+1. \end{cases}
$$
\end{definition}
Note that we take in the definition the maximum of $\{v, sv\}$ with respect to the order $\leq_\diamond$; this is well-defined by Lemma \ref{lengthpm1}. We can now iterate this process for arbitrary $w, v \in W$ to get their Demazure product $w \ast_\diamond v \in W$.

\begin{definition} \label{Demproddef}
Let $w, v \in W$ and fix a reduced expression $w=s_k s_{k-1} \cdots s_1$. Then we define 
$$
w \ast_\diamond v := s_k \ast_\diamond( s_{k-1} \ast_\diamond ( \cdots (s_2 \ast_\diamond (s_1 \ast_\diamond v) ) \cdots ) ) \in W.
$$
\end{definition}

By convention, we set $e \ast_\diamond v:= v$. It is not clear from the definition that $w \ast_\diamond v$ is independent of the reduced expression for $w$; however, the following proposition gives an alternative construction of $w \ast_\diamond v$ that shows that it is indeed independent. Here, we make repeated crucial use of the diamond lemma.

\begin{proposition} \label{Demprodmax} For any $w, v \in W$, $\op{max}_{\leq_\diamond} \{[e,w]v\}$ exists and is equal to $w \ast_\diamond v$, where $[e,w] \subset W$ is the interval with respect to the standard Bruhat order. In particular, $w \ast_\diamond v$ is well-defined. 
\end{proposition}

\begin{proof}
We induct on $l(w)$, the standard length of $w$. The case $l(w)=0$ is clear by convention and the case $l(w)=1$ holds by Definition \ref{simpleDemprod}. Assume that it holds for $l(w)=k$, and suppose that $l(w)=k+1$. Write $w=sw'$ for some simple reflection $s$ and $w' \in W$ with $l(w')=k$. Then by definition, 
$$
w \ast_\diamond v = s \ast_\diamond (w' \ast_\diamond v)
$$
applying Definition \ref{Demproddef} for a particular reduced expression for $w$ (and hence also $w'$). By induction, $w' \ast_\diamond v$ does not depend on the choice of reduced expression and can be written as 
$$
w' \ast_\diamond v = x_0 v 
$$
for some $x_0 \in [e, w']$. Note that, for any $x \in [e, w']$, we have $xv \leq_\diamond x_0 v$. We consider two cases: 
\begin{enumerate} 
\item Suppose $x_0v <_\diamond sx_0v$, so that $w \ast_\diamond v= sx_0v$. Since $x_0 \leq w'$ and $w' < sw'=w$, we have by the diamond lemma (a) applied to $x_0 < w$ that $sx_0 \leq sw'=w$. Now consider any $x \leq w=sw'$. Then again by the diamond lemma (b) applied to $x \leq w$, either $x \leq w'$ or $sx \leq sw$.  In the first case by the induction hypothesis we have $xv \leq_\diamond x_0v$ since $x \in [e,w']$. In the latter case, write $x=sx'$; thus $x' \leq w'$. Then $x'v \leq_\diamond x_0v$. Applying the diamond lemma (a) here for $x'v \leq_\diamond x_0v$, we get that either 
$$
sx'v=xv \leq_\diamond x_0v <_\diamond sx_0v = w \ast_\diamond v,
$$
or else
$$
sx'v=xv \leq_\diamond sx_0v = w \ast_\diamond v.
$$
Thus for any $x \in [e,w]$, we have that $xv \leq_\diamond sx_0v=w \ast_\diamond v$. So, $w \ast_\diamond v = (sx_0)v = \op{max}_{\leq_\diamond} \{[e,w]v\}$. 

\item Suppose $sx_0v <_\diamond x_0v$, so that $w \ast_\diamond v=x_0v$. Take any $x \leq w=sw'$. Applying diamond lemma (b) to $x \leq sw'$ we have either $x \leq w'$ or $sx \leq w'$. In the former case, by the induction hypothesis since $x \in [e, w']$ we have $xv \leq_\diamond x_0v = w \ast_\diamond v$. Else, write $x=sx'$ for $x' \in [e, w']$. Then again $x'v \leq_\diamond x_0v$, and by the diamond lemma (a) for  $x'v \leq_\diamond x_0v$ we get either 
$$
xv=sx'v \leq_\diamond sx_0v <_\diamond x_0v = w \ast_\diamond v,
$$
or else
$$
xv=sx'v \leq_\diamond x_0v =w \ast_\diamond v.
$$
Thus for any $x \in [e,w]$, we have that $xv \leq w \ast_\diamond v = (x_0)v = \op{max}_{\leq_\diamond}\{[e,w]v\}$. 
\end{enumerate}
\end{proof}

\begin{remark}
The existence of $\op{max}_{\le_\diamond}\{[e,w]v\}$ is also proven by Chen and Dyer \cite{CD}, even for more general twisted orders $\le_\eta$. We will examine this in more detail in Section \ref{hummingbirds}. 
\end{remark}

\begin{remark} For the standard Demazure product $\ast$, there is more flexibility with the statement of Proposition \ref{Demprodmax}. That is, it is also known (cf. \cite{BJK}*{Proposition 6.2, Corollary 6.3}, for example) that 
$$
w \ast v = \op{max}_{\leq}\{w [e,v]\}=\op{max}_{\leq} \{xy: x \in [e,w], \ y \in [e,v]\}.
$$
This is \emph{not} the case, however, for $\ast_{-}$ and $\ast_\semi$. In the former, one can define a \emph{left-} and \emph{right-downward Demazure product} \cite{Sad} which are in general different elements. For the latter, one could define a ``right" semi-infinite Demazure product in an analogous way; we do not have need for such a construction here. 
\end{remark}

We conclude this section by recording two additional combinatorial properties of the Demazure products $\ast_\diamond$. For the standard Demazure product, these appear in the literature as the following facts:
\begin{enumerate}
\item $\ast$ is associative: $(w \ast v) \ast u = w \ast (v \ast u)$, and 
\item If $w \ast v = xv$ for the unique $x \in [e,w]$ given by Proposition \ref{Demprodmax}, then $l(w \ast v)=l(x)+l(v).$
\end{enumerate}
However, it is straightforward to check that neither $\ast_-$ nor $\ast_\semi$ satisfy the naive analogues of these properties. Instead, the correct generalization is to consider how $\ast_-$ and $\ast_\semi$ \textit{interact} with $\ast$. As we will not make use of these results, we give them here without proof; these are both straightforward consequences again of the diamond lemma and induction as in the previous proof. 

\begin{proposition} \label{newmonoidpluslengthadd}

\begin{enumerate}
\item For any $w, v, u \in W$, we have 
$$
(w \ast v) \ast_\diamond u = w \ast_\diamond (v \ast_\diamond u);
$$
that is, $\ast_\diamond$ is a left monoid action of $(W, \ast)$ on $W$. 

\item Let $w, v \in W$ and write $w \ast_\diamond v =x_0 v$ for $x_0 \in [e,w]$. Then 
$$
l_\diamond(w \ast_\diamond v) = l(x_0)+ l_\diamond(v),
$$
where as always $l(\cdot)$ is the standard length function. 
\end{enumerate}
\end{proposition}

\begin{remark} For $\ast_-$, the two parts of Proposition \ref{newmonoidpluslengthadd} were first proven by He in \cite{He3} and \cite{He2}, respectively. 
\end{remark}

\section{Geometry of the Affine Weyl Group} \label{geom-orders}

As in Section \ref{affDem}, the positive real affine roots are given by 
$$
\Phi^+_{re} = \{\alpha+k\delta: k>0, \text{ or } k=0\text{ and }\alpha\in \mathring\Phi^+\}.
$$
For an affine root $\beta = \alpha+k\delta$, we use the shorthand $\beta \succ_{+} 0$ (or usually just $\beta \succ 0$) to indicate that $\beta\in \Phi^+_{re}$. 

Similarly, the negative real roots $\Phi^-_{re}$ are defined as $\Phi^-_{re}:=-\Phi^+_{re}$. 
We write $\beta \succ_- 0 $ for $\beta \in \Phi^-_{re}$. 

There is also the set of \emph{semi-infinite positive} real roots, defined by 
$$
\Phi_{\frac{\infty}{2}}^+:=\{\alpha+k\delta: \alpha \in \mathring\Phi^+\};
$$
note that both positive \textit{and} negative real roots can be semi-infinite positive. The shorthand $\beta \succ_{\frac{\infty}{2}}0$ is used to indicate a semi-infinite positive root.

\subsection{Orders and inversions}

The different Bruhat orders can each be described in terms of inverting positive roots $\beta\in \Phi^+_{re}$. For the standard and opposite Bruhat orders this is textbook; for the semi-infinite Bruhat order, see for example \cite{INS}*{\S 2.4, \S A.3}. In the following lemma, $\diamond$ can be substituted with $+, -$, or $\frac{\infty}{2}$ to get a uniform alternative definition of these orders.  

\begin{lemma}[Uniform order definition]\label{uniform}
Let $w\in W$ and $\beta\in \Phi^+_{re}$. Then $w \le_{\diamond} s_\beta w$ if and only if $w^{-1} \beta\succ_\diamond 0$. 
\end{lemma}

\begin{lemma}\label{seahorse}
Let $v,w\in W$ such that $v\le_{\diamond} w$. Let $\eta$ be a coweight such that $\gamma\succ_\diamond 0\implies \langle \gamma, \eta\rangle \ge 0$. Then $v\eta - w\eta$ is a nonnegative combination of simple coroots. 
\end{lemma}

\begin{proof}
It suffices to consider the case $v \le_\diamond w =: s_\beta v$. By Lemma \ref{uniform}, $v^{-1}\beta\succ_\diamond 0$. Hence $\langle v^{-1} \beta, \eta \rangle \ge 0$ and 
$$
v\eta - s_\beta v \eta = \langle \beta, v \eta \rangle \beta^\vee = \langle v^{-1} \beta, \eta \rangle \beta^\vee,
$$
which belongs to the nonnegative span of simple coroots. 
\end{proof}

We will frequently want to identify $W$ with a subset of the flag varieties $G(\mathcal{K})/I$, $G(\mathcal{K})/I^-$, and $G(\mathcal{K})/\tfin U(\mathcal{K})$, and as a preliminary step, with a subset of $G(\mathcal{K})/\tfin$. 
This correspondence comes with a twist on the translations. Specifically, we will map 
$$
w \bt \xi \in \mathring W \ltimes \mathring Q^\vee \mapsto w \gt {-\xi} \in G(\mathcal{K})/\tfin,
$$
where $\gt {-\xi}$ stands for the element $-\xi \in \op{Hom}(\C^*, \tfin) = \tfin(\C[t,t^{-1}])\subset \tfin(\mathcal{K})\subset G(\mathcal{K})$. This twist allows us to neatly align the various Bruhat orders with the geometry of closure relations; see Remark \ref{choice-of-twist}.

\subsection{Positive order}

For $  w \in W$, recall that ${\bf X}(  w)$ denotes the closure of the $I$-orbit through $  w$ in $G(\mathcal{K})/I$: 
$$
{\bf X}(  w) = \overline{I   w I} \subseteq G(\mathcal{K})/I
$$

The following lemma is the most direct analogue of the relationship between geometry and Bruhat order in the finite-dimensional setting. Its proof is omitted, though it serves as the archetype for Lemmas \ref{chopin} and \ref{liszt}. 

\begin{lemma}
Suppose $\beta = \alpha+k\delta\succ_{+}0$ and $  w \le s_\beta   w$. Then 
$$
{\bf X}(  w) \subseteq {\bf X} (s_\beta   w).
$$
\end{lemma}

\subsection{Negative order}

Let ${\bf Y}(w)$ denote the closure of the $I$-orbit through $w$ in $G(\mathcal{K})/I^-$: 
\[{\bf Y}(w) = \overline{I w I^-} \subseteq G(\mathcal{K})/I^-.\] 
Note that this is infinite-dimensional. 

\begin{lemma}\label{chopin}
	Suppose $\beta = \alpha+k\delta \succ_+ 0$ and $w \leq_- s_\beta w$. Then \[{\bf Y}(w) \subseteq {\bf Y}(s_{\beta}w).\] 
\end{lemma}

\begin{proof}
	Note that if $w \leq_-s_{\beta}w$ then we have $w \geq s_{\beta}w$ in the standard Bruhat order. Thus there is some positive root $\gamma$ such that $w \gamma=-\beta$. Using the well-known $SL_2$ formula we obtain, for $a\in \C^*$, \[Is_{\beta}wI^{-}/I^{-}\]\[=IU_{\beta}(a)U_{-\beta}(-a^{-1})U_{\beta}(a)\beta^\vee(a)wI^{-}/I^{-}\]\[=IU_{-\beta}(-a^{-1})U_{\beta}(a)wI^{-}/I^{-}.\] Since $U_{\beta}(a)w=wU_{-\gamma}(a)$, we finally obtain \[Is_{\beta}wI^{-}/I^{-}=IU_{-\beta}(-a^{-1})wI^{-}/I^{-}.\] We obtain the desired result after taking the limit $a \rightarrow \infty$.
\end{proof}	

\subsection{Semi-infinite order} 

For $\tilde w\in W$, let ${\bf Q}(\tilde w)$ denote the closure of the $I$-orbit through $\tilde w$ in $G(\mathcal{K})/
\tfin
U(\mathcal{K})$: 
$$
{\bf Q}(\tilde w) = \overline{I \tilde w 
\tfin
U(\mathcal{K})} \subseteq G(\mathcal{K})/
\tfin
U(\mathcal{K})
$$

\begin{lemma}\label{liszt}
Suppose $\beta = \alpha+k\delta \succ_{+}0$ and $\tilde w \le_{\frac{\infty}{2}} s_\beta \tilde w$. Then 
$$
{\bf Q}(\tilde w) \subseteq {\bf Q}(s_\beta \tilde w). 
$$
\end{lemma}

on Schubert varieties in $G(\mathcal{K})/\tfin U(\mathcal{K})$.

\begin{proof}
It suffices to check that $\tilde w$ can be obtained as a limit point of the set $Is_\beta \tilde w 
\tfin
U(\mathcal{K})$. Write $\tilde w = w \gt{\xi}$, with $w\in \mathring W$ and $\xi \in \mathring Q^\vee$. By multiplying on the right by $\gt{-\xi}$, it is equivalent to verify that $w$ is a limit point of $Is_\beta w 
\tfin
U(\mathcal{K})$. Finally, it is equivalent to verify that $e$ is a limit point of 
$$
w^{-1}Iw s_{w^{-1}\beta} 
\tfin
U(\mathcal{K}),
$$
and that is ultimately what we shall do. 

Now $\beta = \alpha + k \delta$, where either $k=0$ and $\alpha\succ_+ 0$ or $k>0$. So $w^{-1}\beta = w^{-1}\alpha + k\delta$, and 
$$
s_{w^{-1}\beta} = s_{w^{-1}\alpha} \bt{kw^{-1}\alpha^\vee} \mapsto s_{w^{-1}\alpha} \gt{-kw^{-1}\alpha^\vee}
$$
under the identification as an element of $G(\mathcal{K})/\tfin$. Thus we will show that $e$ is a limit point of 
$$
w^{-1}Iw s_{w^{-1}\alpha} \gt{-kw^{-1} \alpha^\vee} 
\tfin
U(\mathcal{K}). 
$$

Let $a\in \C^*$ be arbitrary. Consider the element $U_\alpha(at^k)$, which belongs to $I$ due to the assumptions on $\alpha,k$. We claim that 
$$
wU_\alpha(a t^k) w^{-1} s_{w^{-1}\alpha} \gt{-kw^{-1}\alpha^\vee} a^{-w^{-1}\alpha^\vee} U_{w^{-1}\alpha^\vee} (a t^k)
$$
has limit $e$ as $a\to \infty$. This product exists entirely in the root embedding of $SL_2$ or $PGL_2$ determined by the positive root $w^{-1}\alpha$. In $SL_2$, the calculation is straightforward:

$$
\left[
\begin{array}{rr}
1 & a t^k \\ 
0 & 1
\end{array}
\right]
\left[
\begin{array}{rr}
0 & -1 \\ 
1 & 0
\end{array}
\right]
\left[
\begin{array}{lr}
t^{-k} & 0 \\ 
0 & t^k
\end{array}
\right]
\left[
\begin{array}{lr}
a^{-1} & 0 \\ 
0 & a
\end{array}
\right]
\left[
\begin{array}{rr}
1 & a t^k \\ 
0 & 1
\end{array}
\right]=
\left[
\begin{array}{cr}
1 & 0 \\ 
a^{-1} t^{-k} & 1
\end{array}
\right].
$$
\end{proof}

\begin{remark}\label{choice-of-twist}
Correlating $s_0$ with the fixed point $s_\theta t^{\theta^\vee}$ sets up the geometry so that ${\bf X}(s_0)$ has dimension $1$ and ${\bf Q}(s_0)\subset {\bf Q}(e)$ with ``relative dimension'' $1$. At this point, in the  $\mathring{W} \ltimes \mathring{Q}^\vee$ realization of $W$ one could either choose:
\begin{enumerate}

\item $s_0 \sim s_\theta {\bf t}_{\theta^\vee}$ and $l_{\frac{\infty}{2}}(w {\bf t}_\xi) = l(w) - \langle 2\rho, \xi\rangle$ or

\item  $s_0 \sim s_\theta {\bf t}_{-\theta^\vee}$ and $l_{\frac{\infty}{2}}(w {\bf t}_\xi) = l(w) + \langle 2\rho, \xi\rangle$.

\end{enumerate}
We choose the second so that the length function matches the literature more closely. 
\end{remark}

\section{Reminder on GIT and inequalities} \label{GITsection}

The presence of a nontrivial weight space $V_\lambda^w(\mu)\ne 0$ is 
equivalent to the existence of $\tkm$-invariant global sections of a line 
bundle over ${\bf X}(w)$.  First, let us recall some of the main features from geometric invariant theory, which addresses existence of invariant global sections. 

\subsection{GIT}
Let $X$ be a projective, normal $S$-variety, where $S\simeq \C^*\times \cdots \times \C^*$ is an algebraic torus. Let $\mathbb{L}$ be an $S$-linearized line bundle on $X$. 

\begin{definition}
A point $x\in X$ is \emph{semistable} with respect to $\mathbb{L}$ if for some $n>0$, there is a section $\sigma \in H^0(X,\mathbb{L}^{\otimes n})^S$ such that $\sigma(x)\ne 0$. The set of all semistable points of $X$ is denoted $X^{ss}(\mathbb{L})$, or just $X^{ss}$ if the reference to $\mathbb{L}$ is clear, and any point of $X$ which is not semistable is \emph{unstable}. 
\end{definition}

Given $x\in X$ and a one-parameter subgroup (OPS) $\eta: \C^*\to S$, the stability of $\mathbb{L}$ at $x$ can be probed using $\eta$. Since $X$ is projective, the limit point $x_0=\lim_{t\to 0} \eta(t).x$ exists and is invariant under the action of $\eta$. The fibre $\mathbb{L}_{x_0}$ carries a $\C^*$-action via $\eta$: 
$$
\eta(t).z = t^rz
$$
for some integer $r$. Moreover, if there exists an invariant section $\sigma$ such that $\sigma(x)\ne 0$, then one can show that $r\ge0$. 

\begin{definition}
In the above setting, we define $\mu^{\mathbb{L}}(x,\eta):=-r$. 
\end{definition}

By this definition, a given $x\in X$ being semistable implies $\mu^{\mathbb{L}}(x,\eta)\le0$ for any OPS $\eta$. The Hilbert-Mumford criterion \cite{GIT} asserts that the converse is true: 
$$
\text{$x$ is semistable w.r.t. $\mathbb{L}\iff \mu^{\mathbb{L}}(x,\eta)\le0$ for every OPS $\eta$.}
$$

Hesselink \cite{H} showed that $X\setminus X^{ss}$ has a stratification in which the (unstable) points of each stratum have a common OPS witnessing the instability. In particular, if $X$ is irreducible and has no semistable points, there is a unique open dense stratum $U\subseteq X$ and an OPS $\eta$ such that 
$$
x\in U \implies \mu^{\mathbb{L}}(x,\eta)>0. 
$$

\subsection{Semistability on ${\bf X}(w)$} As recalled in Proposition \ref{affine-borel-weil} above, the sections of line bundles on ${\bf X}(w)$ provide a  realization of the Demazure modules: 
$$
H^0({\bf X}(w), L_w(\lambda))^* \simeq V_\lambda^w. 
$$
We are interested in the weight spaces of $V_\lambda^w$ with respect to the torus $\tkm= \tfin \times \C^*(d)\times \C^*(K)$. 
As in \cite[Lemma 5.6]{BJK}, a rational weight $\mu$ belongs to $P_\lambda^w$ if and only if 
$$
H^0({\bf X}(w), (\C_\mu \otimes L_w(\lambda))^{\otimes n})^{\tkm}\ne0
$$
for some integer $n\ge 1$. Here $\C^*(K)$ acts trivially on $G(\mathcal{K})/I$, and by a scalar on the fibres.

Given that $\langle \lambda, K\rangle = \langle \mu, K\rangle$, the existence of a $\tkm$-semistable locus for the line bundle $\mathbb{L} = \C_\mu \otimes L_w(\lambda)$ precisely controls whether $\mu \in P_\lambda^w$. The numerical criteria for semistability require us to calculate 
$$
\mu^{\mathbb{L}}(x,\eta)
$$
as a function of $x\in {\bf X}(w)$ and $\eta:\C^*\to \tkm$, as we now make explicit.

\subsection{Three families of cocharacters}

The affine Weyl group $W$ acts on the cocharacter lattice $X_*(\tkm)$; however, unlike in the classical setting, there is more than one fundamental domain for this action. In fact, there are infinitely many fixed points given by the multiples of $K$. Even in $X_*(\tkm)/(K)$ there are three fundamental domains. We will use $\delta$ to distinguish among these, and in each case define subgroups $M(\eta)\subseteq P(\eta)\subseteq G(\mathcal{K})$ that extend the roles of Levi and parabolic subgroups. 

\begin{definition}A cocharacter $\eta:\C^*\to \tkm$ is called 
\begin{enumerate}
\item  \emph{positive} if $\langle \delta, \eta\rangle >0$. In this case, one may find $v\in W$ such that $v\eta$ is dominant. 

We define $P(\eta)$ to be the subgroup generated by $v^{-1}Iv$ and the affine root subgroups $U_\alpha$ such that $\langle \alpha,\eta\rangle = 0$. We define $M(\eta)$ to be the subgroup generated by these root subgroups $U_\alpha$.

\item \emph{negative} if $\langle \delta, \eta \rangle <0$. In this case, one may find $v\in W$ such that $v\eta$ is antidominant. 

We associate to $\eta$ the subgroup $P(\eta)$ which is generated by $v^{-1}I^-v$ 

and the affine root subgroups $U_\alpha$ such that $\langle \alpha, \eta \rangle = 0$. 
We define $M(\eta)$ to be the subgroup generated by these root subgroups $U_\alpha$.

\item \emph{level-zero} if $\langle \delta, \eta \rangle = 0$. In this case, one may find $v\in \mathring W$ such that $\langle \alpha_i, v\eta \rangle \ge 0$ for all $i = 1,\hdots, r$. That is, $v \eta$ is a finite dominant coweight. 

We associate to $\eta$ the subgroup $P(\eta)$ generated by $v^{-1}\left(\tfin U(\mathcal{K})\right)v$ and the (real) affine root subgroups $U_\alpha$ such that $\langle \alpha, \eta \rangle = 0$. 
We set $M(\eta)$ to be the subgroup generated by the root subgroups $U_\alpha(\mathcal{K})$ where $\alpha$ is a finite root such that $\langle \alpha,\eta \rangle = 0$. 

\end{enumerate}
Note that the property of being positive, negative, or level-zero is unchanged under the action $\eta\mapsto v\eta$ for $v\in W$. 
\end{definition}

Using the decompositions of Section \ref{geom-orders}, taking inverses, and using that $P(\eta)$ contains (a Weyl conjugate of) either $I$, $I^-$, or $TU(\mathcal{K})$, we obtain the following common decomposition of $\op{Fl}_G$:

\begin{lemma}\label{adequate}
For $\eta$ of any type (positive, negative, level-zero), $\op{Fl}_G$ breaks up as a disjoint union (indexed over an appropriately chosen subset of $W$)
$$
\op{Fl}_G = \bigsqcup_q P(\eta) q I/I.
$$
\end{lemma}

Let $\widehat{I}$ denote the preimage of $I \subset LG$ in $\widehat{G}$. A more natural definition of affine Schubert cells (and their generalizations) would be $P(\eta)q \widehat{I}/\widehat{I}$. We have an isomorphism between $P(\eta)q \widehat{I}/\widehat{I}$ and $P(\eta)qI/I$, and the same holds for their closures, so for geometric questions we may (and do) work with the latter. But the former is more useful for constructing our line bundles on $\mathbf{X}(w)$ via the Borel associated bundle construction.

We write \[L_{\mathbf{X}(w)}(\mathbb{C}_{-\lambda}):= \overline{I w\widehat{I}} \times_{\widehat{I}} \mathbb{C}_{-\lambda}.\] By \cite{KumarBook} Lemma 8.1.4, this line bundle on $\mathbf{X}(w)$ coincides with the pullback of the previously described line bundle: $L_{\mathbf{X}(w)}(\mathbb{C}_{-\lambda})= L_w(\lambda)$.

Now let $x\in {\bf X}(w)$, $\eta:\C^* \to \tkm$, and $\mathbb{L} = \C_\mu\otimes L_w(\lambda)$ as above. We will calculate $\mu^{\mathbb{L}}(x,\eta)$, in the same manner as \cite{BeSj}*{\S 4.2} (see also \cite{BK}*{Lemma 14}) in the finite setting.  We can find some $q\in W$ such that $x\in P(\eta)q\widehat{I}/\widehat{I}$. Moreover, $x_0 = \lim_{t\to 0} \eta(t).x$ belongs to $M(\eta) q\widehat{I}/\widehat{I}$. Write $x_0 = \ell_0q\widehat{I}/\widehat{I}$. The fibre $\mathbb{L}_{x_0}$ can be identified as the collection of pairs modulo a relation:
$$
\mathbb{L}_{x_0} = \{
(\ell_0 q b,z): b\in \widehat I, z\in \C 
\}/(\ell_0 q b,z) \sim (\ell_0 q bb',  \lambda(t') z),
$$
where $t'$ is the part of $b'\in \widehat I$ belonging to $\tkm$.  

Recalling that the $\tkm$-linearization of $\mathbb{L}$ has been twisted by $\mu$, the action of $\eta(t)$ on this line is evidently by 
$$
\lambda^{-1}(q^{-1}\eta(t))\mu(\eta(t)) = t^{\langle -\lambda , q^{-1}\eta \rangle + \langle \mu, \eta \rangle}. 
$$

Therefore, keeping in mind the isomorphism $P(\eta)q\widehat{I}/\widehat{I} \simeq P(\eta)qI/I$, we have calculated the following. 
\begin{lemma}\label{bridge}
Suppose $x\in P(\eta) qI/I$, $\eta:\C^* \to \tkm$, and $\mathbb{L} = \C_\mu\otimes L_w(\lambda)$. Then 
$$
\mu^{\mathbb{L}}(x,\eta) = \langle \lambda, q^{-1} \eta \rangle  - \langle \mu, \eta \rangle.
$$
\end{lemma}

\section{Inequalities for $P_{\lambda}^w$} \label{ineqdetermined}

We produce inequalities defining $P_{\lambda}^w$ in the following fashion. If $\mu \notin P_{\lambda}^w$, this corresponds with $\mathbb{L}=L_w(\lambda) \otimes \mathbb{C}_{-\mu}$ being an unstable bundle on $\mathbf{X}(w)$. Assume that this is the case. Then to each unstable point $x \in {\bf X}(w)$, we can associate the Kempf ``maximal destabilizing OPS" $\eta_x$. Moreover, by work of Hesselink, there will exist a dense open subset $U \subset \mathbf{X}(w)$ such that for all $x,y \in U$ we have $\eta_x=\eta_y$. We can then conjugate this Kempf maximal destabilizing OPS attached to $U$ to a ``standard" OPS $\eta$: $\eta_x=v\eta$. Then $P(\eta)$ will contain either $I$, $I^{-}$, or $\tfin U(\mathcal{K})$, and we can apply Lemmas \ref{adequate} and \ref{bridge}.

\begin{definition}
 We take $B(\eta)$ to be $I$ if $\eta$ is positive, $I^-$ if $\eta$ is negative, and $\tfin U(\mathcal{K})$ if $\eta$ is level-zero. 
 \end{definition}

In order to produce inequalities, we require the following result, which is a straightforward generalization of \cite{BJK}*{Lemma 5.8} to the current setting.

\begin{lemma}\label{some-dense}
	Fix an affine Schubert variety $\mathbf{X}(w)$, and let $\eta$ be a standard OPS. Then there exists $u \in W$ such that $vB(\eta)v^{-1}uI/I \cap {\bf X}(w)$ is dense inside $\mathbf{X}(w)$. 
\end{lemma}

The question becomes how to characterize when this dense intersection occurs. It will be helpful to have a summary of how Schubert cells (in each of the three flag varieties) behave when twisted by simple reflections. 

\begin{lemma}\label{bruhat-does-it-again}
Let $w\in W$ and $s_i$ a simple reflection. Let $\eta:\C^*\to \tkm$ be either dominant, 
antidominant, or finite dominant. Then
$$
s_iIwB(\eta)\subseteq 
\left\{
\begin{array}{lr}
 Is_iwB(\eta), & w \le_{\diamond} s_iw\\
 Is_iwB(\eta) \sqcup IwB(\eta), & s_iw \le_{\diamond}w
\end{array}
\right.
$$
\end{lemma}

\begin{proof}
For $b\in I$, we have 
$$
s_ib = b'  s_i U_{\alpha_i}(c)
$$
for some $b'\in I$ and $c\in \C$. 

In case $w\le_{\diamond} s_iw$, we know that $w^{-1} \alpha_i \succ_\diamond 0$; hence 
$$
s_ib w B(\eta) = b' s_i U_{\alpha_i}(c) w B(\eta) = b' s_i w U_{w^{-1}\alpha_i} (c) B(\eta) = b' s_i w B(\eta) \in I s_i w B(\eta). 
$$

Otherwise, $s_i w \le_{\diamond} w$ and $-\left(w^{-1}\alpha_i\right) \succ_\diamond 0$. If $c=0$, then we still have  
$$
s_i b w B(\eta) = b' s_i w B(\eta) \in I s_i w B(\eta).
$$
But if $c\ne 0$, we observe that 

$$
s_i U_{\alpha_i}(c) = U_{\alpha_i} (-c^{-1}) U_{-\alpha_i}(c)  \pmod {\tfin}
$$
from the $SL_2$ theory. Thus 
\begin{align*}
s_ibwB(\eta) = b' U_{\alpha_i}(-c^{-1}) U_{-\alpha_i}(c)w B(\eta) &= b' U_{\alpha_i}(-c^{-1}) w U_{-w\alpha_i}(c)B(\eta)  \\ 
& =b' U_{\alpha_i}(-c^{-1}) w B(\eta) \in I w B(\eta).  \qedhere
\end{align*}
\end{proof}

Now we can relate dense intersections with closure relations and Demazure products for the various orders $\le_\diamond$. 

\begin{lemma}\label{ankylosaurus}
	The following are equivalent:
	\begin{enumerate}
		\item $\mathbf{X}(w) \subseteq \overline{v B(\eta) v^{-1}uI/I}$
		\item For $B(\eta)$ of type $\diamond$ we have  $w^{-1} \ast_{\diamond} v \leq_{\diamond} u^{-1}v$.
	\end{enumerate}
	Moreover, if  $v B(\eta)v^{-1}uI/I \cap \mathbf{X}(w)$ is dense in $\mathbf{X}(w)$, then the equivalent statements above are true. 
\end{lemma}

\begin{proof}

	(1) $\implies$ (2):  Let $q\le w$ in the standard Bruhat order. Since ${\bf X}(q)\subseteq \overline{v B(\eta) v^{-1}uI/I}$, we see that in particular $v^{-1}q \in \overline{B(\eta) v^{-1}uI/I}$. This is equivalent to $q^{-1}v\in \overline{I u^{-1}v B(\eta)}$ in $G(\mathcal{K})/B(\eta)$, so by the relevant Lemma in Section \ref{geom-orders}, $q^{-1}v \le_{\diamond} u^{-1}v$. Since this is true for arbitrary $q\le w$, we obtain 
	$$
	w^{-1}*_{\diamond}v \le_{\diamond} u^{-1}v
	$$
	by Proposition \ref{Demprodmax}.

	(2) $\implies$ (1): We will show that $v^{-1}{\bf X}(w)\subseteq \overline{B(\eta) (w^{-1}*_{\diamond}v)^{-1}I/I}$. It is equivalent to show $B(\eta) v^{-1}IwI/I\subseteq \overline{B(\eta) (w^{-1}*_{\diamond}v)^{-1}I/I}$, or even 
	\begin{align}\label{geom-dem-prod}
	w^{-1}IvB(\eta) \subseteq \overline{I(w^{-1}*_{\diamond}v)B(\eta)},
	\end{align}
	which is the direction we will take, inducting on $\ell(w)$. 
	
	For $\ell(w) = 0$, the statement is obvious. Now suppose (\ref{geom-dem-prod}) holds whenever $\ell(w) = k$, and consider an arbitrary $w\in W$ with $\ell(w) = k+1$. Decompose $w = w_1s$, where $s\in S$ and $\ell(w_1) = k$. By hypothesis, 
	$$
	w^{-1}Iv B(\eta) = sw_1^{-1}IvB(\eta) \subseteq s \overline{I (w_1^{-1}*_{\diamond} v)B(\eta)}. 
	$$
	By Lemma \ref{bruhat-does-it-again}, 
	$$
	sI (w_1^{-1}*_{\diamond} v)B(\eta) \subseteq \overline{I (s*(w_1^{-1}*_{\diamond} v))B(\eta)} = \overline{I(w^{-1}*_{\diamond} v) B(\eta)}.
	$$
	The result follows, establishing (1) $\iff$ (2). 
	
	Finally, suppose $v B(\eta)v^{-1}uI/I \cap \mathbf{X}(w)$ is dense in $\mathbf{X}(w)$. Then 
	\begin{align*}
	\mathbf{X}(w) = \overline{v B(\eta)v^{-1}uI/I \cap \mathbf{X}(w)} \subseteq \overline{v B(\eta)v^{-1}uI/I}.  & \qedhere
	\end{align*}
\end{proof}	

We are now in a position to prove our first main result. 

\begin{theorem}\label{first-pass}
Let $\mu$ be a rational character. Then $\mu\in P_\lambda^w$ if and only if $\langle \lambda,K\rangle = \langle \mu,K\rangle$ and

\begin{enumerate}
\item for every fundamental coweight $\check\Lambda_i$ and $v\in W$, the inequality 
$$
\langle \lambda, (w^{-1}*v) \check\Lambda_i \rangle \le \langle \mu, v \check\Lambda_i\rangle 
$$
holds, and

\item for every negative fundamental coweight $-\check\Lambda_i$ and $v\in W$, the inequality 
$$
\langle \lambda, (w^{-1}*_{-}v) (-\check \Lambda_i) \rangle \le \langle \mu, v (-\check \Lambda_i) \rangle
$$
holds, and 

\item for every finite fundamental coweight $\check \omega_i$ and $v\in W$, the inequality 
$$
\langle \lambda, (w^{-1}*_{\frac{\infty}{2}}v) \check \omega_i \rangle \le \langle \mu, v \check \omega_i \rangle
$$
holds. 
\end{enumerate} 
\end{theorem}

\begin{proof}
First, suppose $\mu\in P_\lambda^w$. Since $\mu$ is a convex combination of the vertices $q\lambda$, $q\le w$, it will suffice to ensure that each vertex $q\lambda$ satisfies the inequalities (1) - (3), the condition $\langle \lambda, K\rangle = \langle q \lambda , K \rangle$ being apparent from the $W$-invariance of $K$. 

For this, we must argue that 
$$
\langle \lambda, (w^{-1}*_{\diamond} v) \eta \rangle \le \langle q \lambda, v \eta \rangle
$$
for each choice of $\diamond$, where $\langle \gamma, \eta \rangle \ge 0$ for all $\gamma \succ_\diamond 0$. 

From Proposition \ref{Demprodmax}, we know that $q^{-1}v \le_{\diamond} w^{-1}*_{\diamond} v$. Hence 
$$
q^{-1}v \eta - (w^{-1}*_{\diamond} v) \eta
$$
is a nonnegative combination of positive coroots by Lemma \ref{seahorse}. Since $\lambda$ is dominant, the inequality holds. 

Second, suppose $\mu\not\in V_\lambda^w$. It could be that $\langle \lambda, K\rangle \ne \langle \mu, K\rangle$. But suppose $\langle \lambda, K\rangle = \langle \mu, K\rangle$; we will find an inequality that $\mu$ violates.  For all $n\ge 1$, 
$$
H^0({\bf X}(w), (\C_\mu \otimes L_w(\lambda))^{\otimes n})^\tkm=0, 
$$
or in other words, $\mathbb{L} = \C_\mu \otimes L_w(\lambda)$ is unstable. 

There is an open subset $U \subseteq {\bf X}(w)$ with the same OPS $\eta'$ that witnesses the instability: $x\in U \implies \mu^{\mathbb{L}}(x,\eta') >0$. Now, depending on whether $\langle \delta, \eta' \rangle$ is positive, negative, or zero, we can manage to find a $v\in W$ such that $\eta = v^{-1}\eta'$ is dominant, antidominant, or finite dominant, respectively. 

As pointed out in Lemma \ref{some-dense}, we can find $u\in W$ such that $v B(\eta) v^{-1}u I/I$ has dense intersection with ${\bf X}(w)$. So take any $x\in U\cap v B(\eta) v^{-1}u I/I$. We have $\mu^{\mathbb{L}}(x,\eta) = \langle \lambda, u^{-1}v \eta \rangle - \langle \mu, v \eta \rangle $ by Lemma \ref{bridge}, and this expression is $>0$: 
\begin{align*}
\langle \lambda, u^{-1}v\eta \rangle > \langle \mu, v\eta \rangle.
\end{align*}
Recalling Lemma \ref{ankylosaurus}, observe that $u^{-1}v \ge_{\diamond} w^{-1}*_{\diamond} v$, and thus $\langle \lambda, (w^{-1}*_{\diamond} v) \eta \rangle > \langle \lambda, u^{-1}v \eta \rangle$ by Lemma \ref{seahorse}.

Finally, we relate the ``failed inequality'' 
$$
\langle \lambda, (w^{-1}*_{\diamond} v) \eta \rangle  > \langle \mu, v\eta \rangle
$$
to one of those in the theorem statement. Even though $\eta$ might not be fundamental (or negative fundamental, or finite fundamental), we can write $\eta$ as a nonnegative combination $\sum c_i \check \Lambda_i$ (or $\sum c_i  (-\check \Lambda_i)$, or $\sum c_i \check \omega_i$, as the case may be) modulo $K$. As we have assumed $\langle \lambda, K\rangle = \langle \mu, K\rangle$, we get 
$$
\langle \lambda, (w^{-1}*_{\diamond} v) \left(\sum c_i \check \Lambda_i\right) \rangle > \langle  \mu, v \left(\sum c_i \check \Lambda_i\right) \rangle,
$$
or equivalent for the other two cases. For some index $i$ such that $c_i\ne 0$, we must therefore have 
$$
\langle \lambda, (w^{-1}*_{\diamond} v) \check \Lambda_i \rangle > \langle \mu, v  \check \Lambda_i \rangle 
$$ 
or equivalent. 
\end{proof}

\section{Stabilizers of cocharacters and coset representatives} \label{Stabilizer section}

In Theorem \ref{first-pass}, the inequalities describing $P_\lambda^w$ are given in terms of pairings of the form $\langle \nu, v \eta \rangle$ for certain characters $\nu$, cocharacters $\eta$, and Weyl group elements $v$. In the subsequent sections, we will explore the face structure of $P_\lambda^w$. We first introduce some simplifying combinatorial conventions. 

Specifically, the pairings of Theorem \ref{first-pass} are unchanged when replacing $v$ with $u = vx$ as long as $x\eta=\eta$, and in order to understand the essential role of the constraints, we will soon find it helpful to make a judicious choice of a representative $u$. This is done in Corollary \ref{ineqcosetrep} below. First, we analyze some intricacies of the stabilizer subgroup $W_\eta \subset W$. 

When $\eta$ is in the Tits cone, the stabilizers $W_\eta$ and coset representatives $W^\eta:=W/W_\eta$ are textbook. However, the case when $\langle \delta, \eta \rangle=0$ is less widespread in the literature. In this case, the stabilizers are less well-behaved than their classical counterparts. Instead, we find it useful to recall a different subgroup and corresponding coset representatives--\emph{Peterson's coset representatives} \cite{Pet}, \cite{LS}--as these have the analogous combinatorial properties we will need. 

Throughout, we say $\eta$ is an ``appropriately dominant" cocharacter  if it is either affine dominant, affine antidominant, or finite dominant. By way of motivation, we start with the following proposition for when $\eta$ is affine dominant or antidominant (cf. \cite{Kac}*{Prop. 3.12a}).

\begin{proposition} \label{TitsStab} Let $\eta$ be an affine dominant or antidominant cocharacter, and let $W_\eta \subset W$ be the stabilizer of $\eta$. Then $W_\eta$ is generated by the simple reflections which it contains; that is, if $J:=\{j | s_j(\eta)=\eta \} \subseteq \{0,1,\dots, n\}$ and $W_J:=\langle s_j | j \in J \rangle$, then $W_\eta \cong W_J$.
\end{proposition}

\noindent In this case, $W_\eta$ is a standard parabolic subgroup of $W$. At the level of root systems, we can define 
$$
\Phi^+_J =  \Phi^+ \cap \bigoplus_{i \in J} \Z_{\geq 0} \alpha_i,
$$
which are the positive roots in the corresponding Levi subalgebra of $\mf[g]$ corresponding to $J$. Note that when $J \neq \{0,1,\dots, n\}$, then both $|W_\eta|$ and $|\Phi^+_J|$ are finite. 

We can next consider the cosets $W^\eta:=W/W_\eta$, which are more commonly denoted in the literature as $W^J:=W/W_J$, for such $\eta$. These cosets $vW_\eta$ come with distinguished representatives, the \emph{minimum-length representatives}, which for our purposes are most conveniently characterized by the following classical proposition (cf. \cite{BjBr}*{\S 2.4} or \cite{KumarBook}*{1.3.E Exercise}).

\begin{proposition} \label{TitsMinLen} Let $\eta$, $J$, and $W^\eta$ be as above. Then 
$$
W^\eta = \{v \in W | v\left(\Phi^+_J\right) \subset \Phi^+\}.
$$
\end{proposition}

We now want to consider the case when $\eta$ is finite dominant, so that $\langle \alpha_i, \eta \rangle \geq 0$ for $i \neq 0$ and $\langle \delta, \eta \rangle =0$. Then unlike in Proposition \ref{TitsStab}, the stabilizer $W_\eta$ in this case is no longer a finite parabolic subgroup of $W$. We give a precise description of the stabilizers $W_\eta$ in the subsequent Lemma \ref{stabilizer}. But first, we introduce a distinguished subgroup $(W_J)_{\text{af}} \subset W$ determined by $\eta$ and investigate its relation to the stabilizer $W_\eta$.

Concretely, let $\eta$ be a finite dominant cocharacter and set $J:=\{i | \langle \alpha_i, \eta \rangle = 0\} \subseteq \{1,2,\dots,n\}=:[n]$. Then we can consider the parabolic subgroup $W_J \subset \mathring{W}$ of the finite Weyl group. Let $\Phi^+_J$ be the associated set of positive roots; note that here $\Phi^+_J \subseteq \mathring{\Phi}^+$. Let $\mathring{Q}_J^\vee:= \bigoplus_{i \in J} \Z \alpha_i^\vee$ be the associated coroot lattice. With this data, we define the following subgroup of $W$, following the notation of the lecture notes \cite{Pet} (see also \cite{LS}). 

\begin{definition} For $J$ as above, let $(W_J)_{\text{af}}:= W_J \ltimes \mathring{Q}_J^\vee$. 
\end{definition} 

\noindent That is, $(W_J)_{\text{af}}$ is the affine Weyl group corresponding to the finite Weyl group $W_J$. Similarly, we can define an associated set of positive (affine) roots to the index set $J$. 

\begin{definition} Let $\mathring{\Phi}_J$ be the finite root system associated to $J$. Then we define $(\Phi_J)^+_{\text{af}}$ by 
$$
(\Phi_J)^+_{\text{af}}:= \{ \beta \in \Phi^+ | \beta=\gamma+k\delta, \ \gamma \in \mathring{\Phi}_J\}.
$$
\end{definition}

\noindent It is an easy exercise to see that $(W_J)_{\text{af}} \cong \langle s_\beta | \beta \in (\Phi_J)^+_{\text{af}} \rangle$ (cf. \cite{LS}*{Lemma 10.5}). In particular, $(W_J)_{\text{af}}$ is a reflection subgroup of $W$. Moreover, $(W_J)_{\text{af}}$ is actually a Coxeter group, whose simple generators can be computed explicitly via results of Dyer \cite{Dyer1}. We will later be interested in the Coxeter presentation of $(W_J)_{\text{af}}$ and its associated length function. 

The following lemma connects $(W_J)_{\text{af}}$ to the stabilizer of $\eta$. 

\begin{lemma} \label{stabilizer}
Let $\eta$ be finite dominant with $J=\{j | \langle \alpha_j, \eta \rangle =0\}$. Denote by $W_\eta$ the stabilizer subgroup of $\eta$ in $W$. Then $(W_J)_{\text{af}} \subset W_\eta$, and 
$$
W_\eta \cong W_J \ltimes \eta^\perp,
$$
where we define $\eta^\perp:= \{\bt{\xi} | \xi \in \mathring{Q}^\vee, \ (\eta | \xi)=0\}$. Moreover, if $J=[n] \backslash \{i\}$ for some $1 \leq i \leq n$, then $W_\eta \cong (W_J)_{\text{af}}$.
\end{lemma}

\begin{proof}
Using $(W_J)_{\text{af}} \cong \langle s_\beta | \beta \in (\Phi_J)^+_{\text{af}} \rangle$ as remarked above, it is clear that $(W_J)_{\text{af}} \subset W_\eta$, since by construction $\langle \beta, \eta \rangle = 0$ for all $\beta \in (\Phi_J)^+_{\text{af}}$. Now, let $u \in W$ such that $u(\eta)=\eta$. Write $u=w \bt{\xi}$ for $w \in \mathring{W}$ and $\xi \in \mathring{Q}^\vee$. Then we have $\bt{\xi}(\eta)=w^{-1}(\eta)$. Since $\langle \delta, \eta \rangle = 0$, we get 
$$
\bt{\xi}(\eta)= \eta - (\eta | \xi) K.
$$
Pairing with $\Lambda_0$, we see
$$
\langle \Lambda_0, w^{-1}(\eta) \rangle = \langle \Lambda_0, \bt{\xi}(\eta)\rangle = \langle \Lambda_0, \eta \rangle - (\eta | \xi),
$$
as $\langle \Lambda_0, K\rangle =1$. But since $w \in \mathring{W}$, necessarily $\langle \Lambda_0, w^{-1}(\eta) \rangle = \langle \Lambda_0, \eta \rangle$. This forces $(\eta | \xi)=0$, so that $\bt{\xi} \in \eta^\perp$. Finally, $\bt{\xi}(\eta)=\eta$ and thus $w(\eta)=\eta$; viewing $\eta$ as a dominant cocharacter of the finite torus by restriction, we can use the classical result for $\mathring{W}$ to see that $w \in W_J$. This proves the first claim. For $J=[n] \backslash \{i\}$, it is clear to see that in this case $\eta^\perp = \mathring{Q}_J^\vee$. 

\end{proof}

\begin{remark} It is \emph{not} true in general that $W_\eta \cong (W_J)_{\text{af}}$. Consider for example the case when $\mf[g]=A_3^{(1)}$, and let $\eta$ be the cocharacter $\omega_1^\vee+\omega_2^\vee$ and $J=\{3\}$. Then indeed $(W_J)_{\text{af}}$ stabilizes $\eta$, but so do the additional lattice elements $\bt{\xi}$ for $\xi \in \Z(\alpha_1^\vee-\alpha_2^\vee)$.

\end{remark}

While we do not have $W_\eta \cong (W_J)_{\text{af}}$ in all cases, we will see in what follows that it suffices to just consider $(W_J)_{\text{af}}$ and its distinguished coset representatives. We refer to the latter as Peterson's coset representatives, as he first introduced them in his study of the quantum cohomology of generalized flag varieties $G/P$ \cite{Pet}. These representatives are defined to have the same combinatorial properties as minimum-length representatives in the dominant/antidominant cases; more precisely,  

\begin{proposition}[\cite{LS}*{Lemma 10.6}] For $J$ as above, define $(W^J)_{\text{af}}$ by 
$$
(W^J)_{\text{af}}:= \{w \in W | w(\beta) \in \Phi^+ \ \text{for all} \ \beta \in (\Phi_J)^+_{\text{af}}\}.
$$
Then $W=(W^J)_{\text{af}} \ (W_J)_{\text{af}}$; that is, any $w \in W$ can be uniquely written as a product $w=vu$ for $v \in (W^J)_{\text{af}}$, $u \in (W_J)_{\text{af}}$.
\end{proposition}

In conclusion, for an appropriately dominant cocharacter $\eta$ of one of our three types, we make the following definitions to unify notation. 

\begin{definition} \label{etagroupsnotation} Let $\eta$ be an appropriately dominant cocharacter. Then we define the following objects:
\begin{enumerate}
\item subgroups $W(\eta) \subset W$ given by
$$
W(\eta):= \begin{cases} 
W_\eta=W_J, & \eta \text{ is affine dominant or affine antidominant with vanishing } J \\
 (W_J)_{\text{af}}, & \eta \text{ is finite dominant with vanishing } J 
 \end{cases}
 $$
 
 \item the corresponding coset representatives 
 $$
 W^{(\eta)}:= \begin{cases}
 W^\eta=W^J, & \eta \text{ is affine dominant or affine antidominant with vanishing } J \\
 (W^J)_{\text{af}}, & \eta \text{ is finite dominant with vanishing } J 
 \end{cases}
 $$
 
 \item the maps $\pi_{(\eta)}: W \to W(\eta)$ and $\pi^{(\eta)}: W \to W^{(\eta)}$ given by the unique factorizations.
 
 \item subsets of positive roots 
 $$
 \Phi^+_\eta:= \begin{cases}
 \Phi^+_J, & \eta \text{ is affine dominant or affine antidominant with vanishing } J \\
 (\Phi_J)^+_{\text{af}}, & \eta \text{ is finite dominant with vanishing } J 
 \end{cases}
 $$
 so that $W^{(\eta)}:= \{w \in W | w(\Phi^+_\eta) \subset \Phi^+ \}$ in each case. 
 
 \end{enumerate}
 \end{definition}

Finally, we return to the statement of Theorem \ref{first-pass}. Consider as in the proof of Theorem \ref{first-pass} any inequality of the form 
$$
\langle \lambda, (w^{-1} \ast_\diamond v)\eta \rangle \leq \langle q\lambda, v \eta \rangle
$$
for each choice of $\diamond$ and appropriately dominant $\eta$. Here, $q \leq w$ and $v \in W$ was chosen so that the Kempf OPS $\eta' = v\eta$. Of course, we could choose $v\in W^{(\eta)}$ by Lemma \ref{stabilizer}. Thus, we can simplify slightly our considerations in Theorem \ref{first-pass} as below. In the following section on faces of $P_\lambda^w$, this restriction will be combinatorially advantageous. 

\begin{corollary} \label{ineqcosetrep}

In each of the inequalities defining $P_\lambda^w$, we can freely restrict to considering only those $v \in W^{(\eta)}$.
\end{corollary}

\section{Faces of $P_{\lambda}^w$} \label{facesandtwisteds}

We next want to understand the faces of the polytope $P_\lambda^w$. In general, faces are determined by fixing a collection of the defining inequalities of Theorem \ref{first-pass} as equalities. In the finite-dimensional reductive case \cite{BJK}, faces were explicitly described for maximal parabolic subgroups $P_i$--that is, the \emph{facets} of $P_\lambda^w$--although the same analysis holds for all faces. 

Fix $\eta$ an appropriately dominant cocharacter, and denote by $(\leq_\diamond, \ast_\diamond)$ the corresponding Bruhat order and Demazure product related to $\eta$ in the inequalities of $P_\lambda^w$. Fix also an element $v \in W$; by Corollary \ref{ineqcosetrep}, we make a prudent choice of coset representative and assume $v \in W^{(\eta)}$. With this data, we make the following definition for faces of $P_\lambda^w$.

\begin{definition} For $\eta$, $(\leq_\diamond, \ast_\diamond)$, and $v \in W^{(\eta)}$ as above, the face $\mathcal{F}(v, \eta)$ is given by 
$$
\mathcal{F}(v,\eta):= \left\{ \mu \in P_\lambda^w: \langle \lambda, (w^{-1}\ast_\diamond v) \eta \rangle = \langle \mu, v \eta \rangle \right\}.
$$
\end{definition}

We would like to understand the combinatorial structure of these faces. As a first pass, we can derive the following proposition, which relates vertices $q\lambda$, $q \leq w$, lying on a fixed face to certain translates of the stabilizers $W_\eta$. As it will have no effect on $V_\lambda^w$ or $P_\lambda^w$, for technical convenience, we assume that $w$ is the coset representative of \emph{maximal} length in $wW_\lambda$. Such a representative exists, since by assumption $\lambda$ is affine dominant and not a multiple of $\delta$. Therefore $W_\lambda$ is a finite parabolic subgroup. 

\begin{proposition} \label{faces-combinatorics}
Let $\eta, v$, and $\mathcal{F}(v, \eta)$ be as above. A vertex $q\lambda$ belongs to the face $\mathcal{F}(v, \eta)$ if and only if (for some choice of $\bar q\in qW_\lambda$)
$$
(\bar q)^{-1} \in (w^{-1} *_{\diamond}v) W(\eta) v^{-1}.
$$
\end{proposition}

The proposition follows at once from the next three lemmas. 

\begin{lemma}
Suppose $\lambda$ is a dominant weight and $\eta$ an appropriately dominant coweight. Let $w_1 \le_\diamond w_2$. Then 
$$
W_\lambda w_1 W_\eta = W_\lambda w_2 W_\eta \iff \langle \lambda, w_1 \eta \rangle = \langle \lambda, w_2 \eta\rangle
$$
\end{lemma}

\begin{proof}
The implication $\Rightarrow$ is immediate by the $W$-invariance of $\langle, \rangle$. 

For $\Leftarrow$, proceed by induction on $l_\diamond(w_2) - l_\diamond(w_1)$, the base case being trivial in that $w_1 = w_2$. 

For $w_1\ne w_2$, find some positive root $\gamma$ such that $w_1 <_\diamond s_\gamma w_1 \le_\diamond w_2$. Note that $w_1^{-1} \gamma\succ_\diamond 0$, so that $\langle w_1^{-1} \gamma, \eta \rangle \ge 0$. Now 
$$
\langle \lambda, w_1 \eta \rangle \ge \langle \lambda, s_\gamma w_1 \eta \rangle \ge \langle \lambda, w_2 \eta \rangle,
$$
but the outer pairings are assumed to be equal. Thus 
$$
\langle \lambda, w_1 \eta \rangle = \langle \lambda, s_\gamma w_1 \eta \rangle, 
$$
which implies that $\langle \lambda, \gamma \rangle \langle \gamma, w_1 \eta \rangle = 0$, and 
$$
\langle \lambda, s_\gamma w_1 \eta \rangle = \langle \lambda, w_2 \eta \rangle,
$$
so by induction $W_\lambda s_\gamma w_1 W_\eta = W_\lambda w_2 W_\eta$. 

Either $\langle \lambda, \gamma \rangle = 0$, in which case $s_\gamma \in W_\lambda$, or $\langle w_1^{-1} \gamma, \eta \rangle = 0$, in which case $s_{|w_1^{-1}\gamma|} \in W_\eta$. Either way, 
$$
W_\lambda s_\gamma w_1 W_\eta = W_\lambda w_1 W_\eta,
$$
from which the result follows. 
\end{proof}

When $\eta$ is finite dominant, we have seen that $W_\eta$ is unwieldy; we instead prefer to work with the combinatorially-approachable subgroup $W(\eta)$. This change is made possible by the following lemma.

\begin{lemma}\label{smaller-double-cosets}
Suppose $\eta$ is a finite dominant coweight and $\lambda$ is an affine dominant weight with $\langle \lambda, K \rangle >0$. Let $w_1\le_{\frac{\infty}{2}} w_2$. 
If 
$$
W_\lambda w_1 W_\eta = W_\lambda w_2 W_\eta,
$$
then 
$$
W_\lambda w_1 W(\eta) = W_\lambda w_2 W(\eta).
$$
\end{lemma}

\begin{proof}
We may write $w_1 = x w_2 y$, where $x\in W_\lambda$ and $y \in W_\eta$. Let $J = \{ j\ge 1: \langle \alpha_j, \eta\rangle = 0\}$. By Lemma \ref{stabilizer} we may decompose $y = {\bf t}_\xi z$ where $(\eta | \xi) = 0$ and $z \in W_J$. We will show that $\xi \in \oplus_{j\in J} \Z \alpha_j^\vee$, so that $y \in W(\eta)$. 

By the assumption $w_1 \le_{\frac{\infty}{2}} w_2$, we have for any finite fundamental weight $\omega_k^\vee$ that 
$$
\langle \lambda, w_1 \omega_k^\vee \rangle \ge \langle \lambda, w_2 \omega_k^\vee\rangle.
$$
Write $\eta = \sum_{k \not\in J} c_k \omega_k^\vee$, where by definition each $c_k>0$. Since $\langle \lambda, w_1 \eta\rangle = \langle \lambda, w_2 \eta \rangle$, we must have $\langle \lambda, w_1 \omega_k^\vee \rangle = \langle \lambda, w_2 \omega_k^\vee \rangle$ for each $k \not \in J$. Fix an arbitrary $k\not\in J$. From $w_1 = x w_2 {\bf t}_{\xi} z$, we also have 
$$
\langle \lambda, w_1 \omega_k^\vee \rangle = \langle \lambda, w_2 {\bf t}_{\xi} \omega_k^\vee \rangle,
$$
by virtue of $z\in W_J$. Hence 
$$
\langle \lambda, w_2 \omega_k^\vee \rangle = \langle \lambda, w_2 {\bf t}_{\xi} \omega_k^\vee \rangle.
$$
As $\langle \delta, \omega_k^\vee \rangle = 0$, we have (cf. \cite{Kac}*{Eq. (6.5.5)})
$$
{\bf t}_{\xi} \omega_k^\vee = \omega_k^\vee - (\omega_k^\vee | \xi ) K. 
$$
We thus obtain 
$$
\langle w_2^{-1}\lambda, \omega_k^\vee \rangle = \langle w_2^{-1}\lambda, \omega_k^\vee \rangle - (\omega_k^\vee | \xi) \langle w_2^{-1} \lambda, K\rangle ,
$$
or
$$
0 =  (\omega_k^\vee | \xi) \langle w_2^{-1} \lambda, K\rangle .
$$
As $\langle w_2^{-1} \lambda, K\rangle = \langle \lambda, K\rangle >0$, we find that $(\omega_k^\vee | \xi) = 0$ for every $k\not\in J$, as claimed. This implies that $\xi \in \oplus_{j\in J} \Z\alpha_j^\vee$, thus ${\bf t}_\xi \in (W_J)_{\text{af}} = W(\eta)$. 
\end{proof}

\begin{lemma}
If $q\le w$ is such that $W_\lambda q^{-1}v W(\eta) = W_\lambda (w^{-1} *_\diamond v)W(\eta)$, then there is a $\bar q \in q W_\lambda$ such that 
$$
(\bar q)^{-1} \in (w^{-1} *_\diamond v)W(\eta) v^{-1}.
$$
Moreover, assuming $w$ is a maximal length coset representative, $\bar q \le w$. 
\end{lemma}

This concludes the proof of Proposition \ref{faces-combinatorics}. The determination of the face $\mathcal{F}(v,\eta)$ is now captured by the following task:

\begin{task} \label{task-farce}
Characterize the elements $q\le w$ such that 
$$
q^{-1} \in (w^{-1} *_{\diamond}v) W(\eta) v^{-1}. 
$$
That is, determine the intersection
$$
W(\eta) \cap \pi^{(\eta)}(w^{-1} \ast_\diamond v)^{-1} [e, w^{-1}] v.
$$
\end{task}

When $\eta$ is affine dominant, so that $W(\eta)=W_\eta$ is a finite parabolic subgroup and $(\leq_\diamond, \ast_\diamond)=(\leq, \ast)$ are the usual Bruhat order and Demazure product, the approach as in the finite case (\cite{BJK}*{Proposition 7.4}) applies mutatis mutandis and one can find that 
\begin{equation} \label{finiteinterval}
W_\eta \cap \pi^{(\eta)}(w^{-1} \ast v)^{-1} [e,w^{-1}] v = [e, \pi_{(\eta)}(w^{-1} \ast v)]
\end{equation}
as Bruhat intervals in $W_\eta$. However, even in the case when $\eta$ is affine antidominant and again $W(\eta)=W_\eta$, the naive analogue of (\ref{finiteinterval}) for $(\leq_{-}, \ast_{-})$ fails, as seen in the following example; see Corollary \ref{classicintersection} for the correct formulation for this particular class of example. 

\begin{example} 
Let $W$ be the affine Weyl group of type $A_3^{(1)}$, and set $\eta = -\Lambda_1^\vee-\Lambda_3^\vee$ so that $W(\eta):=\langle s_0, s_2\rangle$. Let $w = s_0s_3s_2s_1s_2s_0 \in W$ and fix $v:=e \in W^{(\eta)}$. Since $e$ is maximal with respect to the order $\leq_{-}$, we have that $w^{-1} \ast_{-} v = \pi^{(\eta)}(w^{-1} \ast_{-} v) = \pi_{(\eta)}(w^{-1} \ast_{-} v) = e$. However, we have as intervals in $W(\eta)$ that
$$
W(\eta) \cap [e, w^{-1}] = [e, s_0s_2] \neq \{e\}.
$$
\end{example}

Roughly speaking, the Demazure products $\ast_{-}$ and $\ast_{\semi}$ lose too much information when considering their interaction with $W(\eta)$ and $\pi_{(\eta)}(\cdot)$ for the combinatorial structure on the face $\mathcal{F}(v, \eta)$. Instead, we must utilize a finer combinatorial tool that remembers data from the coweight $\eta$: \emph{twisted Bruhat orders} and their associated Demazure products. These orders were first introduced by Dyer \cite{Dyer3, Dyer4} and have found applications to parabolic Kazhdan--Lusztig theory and the combinatorics of wonderful compactifications \cite{CD}. Using this approach, we complete our Task \ref{task-farce} by constructing a new element $z_\eta \in W$ satisfying $\pi^{(\eta)}(z_\eta)=\pi^{(\eta)}(w^{-1} \ast_\diamond v)$ and relate the corresponding intersection to the interval $[e, \pi_{(\eta)}(z_\eta)]$ in $W(\eta)$; this is, crucially, in general distinct from $[e, \pi_{(\eta)}(w^{-1} \ast_\diamond v)]$.

In the next section, we recall the twisted Bruhat orders associated to cocharacters, their defining length functions, and the related combinatorics. Many of the preliminary results of this section are the appropriate analogues of those for classical Bruhat order, stated uniformly across this more general framework. Then in Section \ref{hummingbirds} we introduce new classes of Demazure products associated to these orders and apply this machinery to understand the combinatorial structure of the faces in $P_\lambda^w$.

\section{Twisted Bruhat orders} \label{orderable-horrors}

We take as our standard references \cite{Dyer3}, \cite{Dyer4}; while twisted Bruhat orders are connected to the more general notions of initial sections of reflection orders, we restrict to the special case of those determined by appropriately dominant cocharacters (cf. \cite{Dyer4}*{(1.2)}).

\begin{definition} 
Let $\eta$ be an appropriately dominant cocharacter. Then we define a set of positive roots $A'_\eta$ and set of reflections $A_\eta$ via
$$
A'_\eta:=\{\beta \in \Phi^+_{re} | \langle \beta, \eta \rangle < 0 \},
$$
and 
$$
A_\eta:= \{s_\beta | \beta \in A'_\eta\}.
$$
\end{definition}

For a Weyl group element $w \in W$, recall the \emph{inversion set} 
$$
\Phi_w:=\Phi^+ \cap w(\Phi^-) = \{ \beta \in \Phi^+ | w^{-1}(\beta) \in \Phi^- \};
$$
note that necessarily $\Phi_w \subset \Phi^+_{re}$. Then it is classically known that $|\Phi_w| = l(w)$ (and similarly $|\Phi_{w^{-1}}| = l(w)$, as $l(w)=l(w^{-1})$). This gives an alternative approach to defining the length function on $W$ that avoids the reliance on reduced words. We can now introduce the twisted length functions (associated to $\eta$) by suitably modifying this relationship between length and inversion sets. 

\begin{definition} 
Let $\eta$ be an appropriately dominant cocharacter and $A'_\eta$ as before. Then the twisted length function $l_\eta: W \to \Z$ is given by 
$$
l_\eta(w):= l(w) - 2 | \Phi_{w^{-1}} \cap A'_\eta|.
$$
\end{definition}

Identically to the case of $l_\semi$ or more generally $l_\diamond$ from Section \ref{Orders and Demazure products}, once we have this length function $l_\eta$ we can again associate to it an ordering on $W$; these are what we refer to as twisted Bruhat orders.

\begin{definition} 
Let $\eta$ be an appropriately dominant cocharacter and $l_\eta(\cdot)$ the associated twisted length function as above. Then the twisted Bruhat order $\leq_\eta$ is a partial order on $W$ defined as follows: for any reflection $s_\beta \in W$ and arbitrary $v \in W$, we say that $v \xrightarrow{\beta} s_\beta v$ if and only if $l_\eta(s_\beta v)=l_\eta(v)+1$. For general $w,v \in W$, we say that $w\leq_\eta v$ if and only if there exist a sequence of reflections $s_{\beta_1},\dots, s_{\beta_k}$ (possibly empty) such that 
$$
w \xrightarrow{\beta_1} s_{\beta_1} w \xrightarrow{\beta_2} s_{\beta_2}s_{\beta_1} w \xrightarrow{\beta_3} \cdots \xrightarrow{\beta_k} s_{\beta_k}\cdots s_{\beta_1}w = v.
$$
\end{definition}

The benefit of these twisted Bruhat orders is that they give a generalization of the three Bruhat orders $\leq, \leq_{-}$, and $\leq_\semi$ on affine Weyl groups. Our definition for $A'_\eta$ differs from that in \cite{Dyer4} by a sign; this is so that $l_\eta(w)=l(w)$ and $\leq_\eta=\leq$ when $\eta$ is affine dominant. Indeed, $A'_\eta=\varnothing$ in that case. Further, if $\eta$ is affine antidominant \emph{regular}, then $A'_\eta=\Phi^+$, so in this case 
$$
l_\eta(w)=l(w)-2|\Phi_{w^{-1}} \cap \Phi^+| = l(w)-2l(w)=l_{-}(w),
$$
and thus $\leq_\eta=\leq_{-}$. Finally, if $\eta$ is finite dominant \emph{regular}--so that $\langle \alpha_i, \eta \rangle >0$ for all $i \neq 0$ and $\langle \delta, \eta \rangle=0$--then the following proposition is due to Dyer \cite{Dyer4}*{Proposition 1.14}.

\begin{proposition} 
For $\eta$ finite dominant regular, $l_\eta(w)=l_\semi(w)$ and $\leq_\eta=\leq_\semi$. 
\end{proposition}

\noindent With this perspective, we can alternatively refer to the three orders up to now denoted $\leq_\diamond$ as the \emph{regular $\eta$ orders}, as these arise collectively when $\eta$ is appropriately dominant and regular. 

When $\eta$ is not regular, the associated length function $l_\eta(\cdot)$ and order $\leq_\eta$ need not be one of the previous orders. Nevertheless, the twisted Bruhat orders $\leq_\eta$ and length functions $l_\eta(\cdot)$ share key properties with $(l_\diamond, \leq_\diamond)$; the first of these is one of the most crucial for our purposes. Namely, $(l_\eta, \leq_\eta)$ satisfies the diamond lemma, as proven by Dyer \cite{Dyer3}*{Proposition 1.9}.

\begin{proposition} \label{twisteddiamond}
Let $\eta$ be an appropriately dominant cocharacter. Then the twisted Bruhat order $\leq_\eta$ satisfies the diamond lemma (as in Lemma \ref{diamond}).
\end{proposition}

Further, the length functions $l_\eta(\cdot)$ can be obtained by a modification of the ``regular" case by remembering data from the vanishing of $\eta$. We make this precise, as recorded in the following proposition.

\begin{proposition} \label{lengthidentification}

Let $\eta$ be an appropriately dominant cocharacter, and let $l_\diamond(\cdot)$ be the regular length function of the same type (that is, $l(\cdot)$, $l_-(\cdot)$, or $l_\semi(\cdot)$ according to $\eta$ affine dominant, affine antidominant, or finite dominant). Then for any $w \in W$,

$$
l_\eta(w)=l_\diamond(\pi^{(\eta)}(w))+l_{W(\eta)}(\pi_{(\eta)}(w)),
$$
where $l_{W(\eta)}(\cdot)$ is the length function on $W(\eta)$ in its Coxeter presentation. 
\end{proposition}

Before coming to a proof, which will focus on the relationship between finite dominant $\eta$ and the semi-infinite length function, we make a few remarks. First, observe that when $\eta$ is affine dominant we have as mentioned above $l_\eta(\cdot)=l(\cdot)$, and also that $W(\eta)=W_\eta$ is a finite standard parabolic subgroup and $W^{(\eta)}=W^\eta$ the minimum-length coset representatives. In this case $l_{W(\eta)}(\cdot)=l(\cdot)$ is the usual length function restricted to $W_\eta$. All together, for any $w \in W$, Proposition \ref{lengthidentification} simply recovers the classical property of minimum-length representatives
$$ 
l(w)=l(\pi^{(\eta)}(w))+l(\pi_{(\eta)}(w)).
$$

When $\eta$ is affine antidominant, again $W(\eta)=W_\eta$ is a finite parabolic subgroup. Here, Proposition \ref{lengthidentification} would say that, for $w \in W$,
\begin{equation} \label{antidomlengthrel}
l_\eta(w)=l_{-}(\pi^{(\eta)}(w))+l_{W(\eta)}(\pi_{(\eta)}(w))=-l(\pi^{(\eta)}(w))+l(\pi_{(\eta)}(w)).
\end{equation}
In this setting, (\ref{antidomlengthrel}) was proven by Dyer \cite{Dyer3}*{Lemma 5.3} (see also \cite{CD}*{Lemma 4.1}). 

Now, let $\eta$ be a finite dominant cocharacter with vanishing $J$. Then $W(\eta)=(W_J)_{\text{af}}$, which as we saw in Section \ref{Stabilizer section} is a reflection group given by $\langle s_\beta | \beta \in (\Phi_J)^+_{\text{af}} \rangle$. Furthermore, this is itself a Coxeter group, although its presentation as such is \emph{not} inherited from the standard Coxeter system for $W$. As all we will need is the associated length function, we record the following lemma from Peterson \cite{Pet}*{Lecture 13}.

\begin{lemma} \label{PetLength}
As a Coxeter group, $(W_J)_\text{af}$ has length function 
$$
l_{(W_J)_\text{af}}(w)= | \{ \beta \in (\Phi_J)^+_{\text{af}} \ | w(\beta) \in \Phi^-\}| = | (\Phi_J)^+_{\text{af}} \cap \Phi_{w^{-1}}|.
$$
\end{lemma}

With this in hand, we can now turn to the proof of Proposition \ref{lengthidentification}.

\begin{proof}[Proof of Proposition \ref{lengthidentification}]
Let $\eta$ be an appropriately dominant cocharacter. For $\eta$ affine dominant or antidominant, see the preceding discussion and references therein. We focus here on $\eta$ finite dominant. To that end, recall by definition that for any $w \in W$, we set 
$$
l_\eta(w):= l(w)-2 | \Phi_{w^{-1}} \cap A'_\eta|,
$$
where $A'_\eta=\{\beta \in \Phi^+_{re} | \langle \beta, \eta \rangle <0 \}$. Let $\eta$ have vanishing $J$, so that $\langle \alpha_i, \eta \rangle =0$ if and only if $i \in J$, $\langle \alpha_i, \eta \rangle >0 $ for $i \not \in J$, and $\langle \delta, \eta \rangle=0$. Then explicitly we have 
$$
A'_\eta =\{ -\gamma+n \delta | \gamma \in \mathring{\Phi}^+ \backslash \mathring{\Phi}^+_J, \ n > 0 \}. 
$$

We partition $\Phi^+_{re}$ into four disjoint sets: 
$$
\begin{aligned}
\Omega_1&:=\{ \alpha+n \delta | \alpha \in \mathring{\Phi}^+_J, \ n \geq 0\}, \\
\Omega_2&:=\{ \alpha+n \delta | \alpha \in \mathring{\Phi}^+ \backslash \mathring{\Phi}^+_J, \ n \geq 0\}, \\
\Omega_3&:=\{ - \alpha+n \delta | \alpha \in \mathring{\Phi}^+_J, \ n >0\}, \\
\Omega_4&:=\{ - \alpha+n \delta | \alpha \in \mathring{\Phi}^+ \backslash \mathring{\Phi}^+_J, \ n >0\}. 
\end{aligned}
$$
Note that we have $A'_\eta = \Omega_4$ and that $\Omega_1 \sqcup \Omega_3 = (\Phi_J)^+_{\text{af}}$. Then we compute $l_\eta(w)$ by 
$$
l_\eta(w) = l(w) - 2| \Phi_{w^{-1}} \cap \Omega_4| = l(w)-2|\Phi_{w^{-1}} \cap (\Omega_3 \sqcup \Omega_4)| +2 |\Phi_{w^{-1}} \cap \Omega_3|.
$$
But by our discussion on \emph{regular} orders--and in particular for finite regular dominant cocharacters--we get that 
$$
l(w)-2| \Phi_{w^{-1}} \cap (\Omega_3 \sqcup \Omega_4)| = l_\semi(w),
$$
thus we have $l_\eta(w)=l_\semi(w)+2|\Phi_{w^{-1}} \cap \Omega_3|$. 

Using the interaction of the unique factorization $W=(W^J)_\text{af} (W_J)_\text{af}$ and $l_\semi(\cdot)$ (cf. \cite{Pet}*{Lecture 13}), we have
$$
l_\semi(w)=l_\semi(\pi^{(\eta)}(w))+l_\semi(\pi_{(\eta)}(w))
$$
so that 
$$
l_\eta(w) = l_\semi(\pi^{(\eta)}(w))+l_\semi(\pi_{(\eta)}(w)) + 2| \Phi_{w^{-1}} \cap \Omega_3|.
$$

Consider now the set $\Phi_{w^{-1}} \cap \Omega_3$. By definition, this is precisely the set of positive real roots $\beta \in \Omega_3$ such that $w(\beta) \in \Phi^-$. This we rewrite as
$$
\pi^{(\eta)}(w) \left(\pi_{(\eta)}(w) (\beta)\right) \in \Phi^-.
$$
But $\pi_{(\eta)}(w)(\beta)$ is of the form $\gamma+k \delta$ for some $\gamma \in \mathring{Q}_J$, since $\pi_{(\eta)}(w) \in (W_J)_\text{af}$ and $\beta \in \Omega_3 \subset (\Phi_J)^+_\text{af}$. Suppose that this is a positive root. Then $\pi_{(\eta)}(w)(\beta) \in (\Phi_J)^+_\text{af}$, forcing 
$$
w(\beta)=\pi^{(\eta)}(w) \left(\pi_{(\eta)}(w) (\beta)\right) \in \Phi^+,
$$
a contradiction. Thus $\beta \in \Phi_{w^{-1}} \cap \Omega_3$ forces $\beta \in \Phi_{\pi_{(\eta)}(w)^{-1}} \cap \Omega_3$. The opposite inclusion $\Phi_{\pi_{(\eta)}(w)^{-1}} \cap \Omega_3 \subseteq \Phi_{w^{-1}} \cap \Omega_3$ is similar. So we get 
$$
 \Phi_{w^{-1}} \cap \Omega_3 = \Phi_{\pi_{(\eta)}(w)^{-1}} \cap \Omega_3.
$$

Thus, we have 
$$
\begin{aligned}
l_\eta(w) &=  l_\semi(\pi^{(\eta)}(w))+l_\semi(\pi_{(\eta)}(w)) + 2| \Phi_{w^{-1}} \cap \Omega_3| \\
&= l_\semi(\pi^{(\eta)}(w))+l_\semi(\pi_{(\eta)}(w)) +2 | \Phi_{\pi_{(\eta)}(w)^{-1}} \cap \Omega_3 | \\
&=  l_\semi(\pi^{(\eta)}(w)) + l_\eta(\pi_{(\eta)}(w)),
\end{aligned}
$$
by applying the connection (valid for any $w$) between $l_\eta(\cdot)$ and $l_\semi(\cdot)$ from above to $\pi_{(\eta)}(w)$. In total, we will have proven Proposition \ref{lengthidentification} if we can show that $l_\eta(\cdot) = l_{(W_J)_\text{af}}(\cdot)$ when restricted to $(W_J)_\text{af}$.

To this end, fix $v \in (W_J)_\text{af}$. Again exploiting the definition of $l_\eta(\cdot)$, we have 
$$
\begin{aligned}
l_\eta(v) &= l(v) - 2 | \Phi_{v^{-1}} \cap \Omega_4| \\
&= |\Phi_{v^{-1}} \cap \Omega_1| + |\Phi_{v^{-1}} \cap \Omega_2 | + |\Phi_{v^{-1}} \cap \Omega_3| - | \Phi_{v^{-1}} \cap \Omega_4|,
\end{aligned}
$$
using the partition of $\Phi^+_{re}$ and that $l(v)=| \Phi_{v^{-1}} \cap \Phi^+_{re}|$. By Lema \ref{PetLength}, we can rewrite this as 
$$
l_\eta(v) = l_{(W_J)_\text{af}}(v) + | \Phi_{v^{-1}} \cap \Omega_2 | - | \Phi_{v^{-1}} \cap \Omega_4|.
$$

Finally, we can conclude the proof by showing for any $v \in (W_J)_\text{af}$ that 
$$
|\Phi_{v^{-1}} \cap \Omega_2| = | \Phi_{v^{-1}} \cap \Omega_4|.
$$
Write $v=y \bt{\xi}$ in the presentation $(W_J)_\text{af}=W_J \ltimes \mathring{Q}^\vee_J$. Fix a positive root $\beta \in \Omega_2$; recall that this means we can write 
$$
\beta=\alpha+n\delta, \ \alpha \in \mathring{\Phi}^+ \backslash \mathring{\Phi}^+_J, \ n \geq 0.
$$
Applying $v$, we get 
$$
v(\beta)= y \bt{\xi}(\alpha+n\delta) = y(\alpha)+ \left(n-\langle \alpha, \xi \rangle \right)\delta.
$$
Since $y \in W_J$ and $\alpha \in \mathring{\Phi}^+ \backslash \mathring{\Phi}^+_J$, $y(\alpha)$ is a positive root. Therefore we have 
$$
\Phi_{v^{-1}} \cap \Omega_2 = \{\beta \in \Omega_2 | v(\beta) \in \Phi^-\} = \{\alpha+n\delta \in \Omega_2 | n- \langle \alpha, \xi \rangle <0\}.
$$
In particular, we get for such an $\alpha+n \delta$ that $0 \leq n < \langle \alpha, \xi \rangle$, and any such $n$ will work. In total, this gives 
$$
| \Phi_{v^{-1}} \cap \Omega_2| = \sum_{\substack{\alpha \in \mathring{\Phi}^+ \backslash \mathring{\Phi}^+_J \\ \langle \alpha, \xi \rangle >0}} \langle \alpha, \xi \rangle.
$$

By similar considerations, 
$$
\Phi_{v^{-1}} \cap \Omega_4 = \{ -\alpha+n\delta \in \Omega_4 | -y(\alpha) + \left( n+ \langle \alpha, \xi \rangle \right) \delta \in \Phi^-\},
$$
forcing $0<n \leq -\langle \alpha, \xi \rangle$ and therefore we have 
$$
|\Phi_{v^{-1}} \cap \Omega_4| = \sum_{\substack{\alpha \in \mathring{\Phi}^+ \backslash \mathring{\Phi}^+_J \\ \langle \alpha, \xi \rangle <0}} - \langle \alpha, \xi \rangle.
$$

In total, we have $|\Phi_{v^{-1}} \cap \Omega_2| = | \Phi_{v^{-1}} \cap \Omega_4|$ if and only if 
$$
\sum_{\substack{\alpha \in \mathring{\Phi}^+ \backslash \mathring{\Phi}^+_J \\ \langle \alpha, \xi \rangle >0}} \langle \alpha, \xi \rangle = \sum_{\substack{\alpha \in \mathring{\Phi}^+ \backslash \mathring{\Phi}^+_J \\ \langle \alpha, \xi \rangle <0}} - \langle \alpha, \xi \rangle,
$$
or equivalently (as including those $\alpha$ with $\langle \alpha, \xi \rangle =0$ does not change the sum) if and only if
$$
\sum_{\alpha \in \mathring{\Phi}^+ \backslash \mathring{\Phi}^+_J} \langle \alpha, \xi \rangle =0.
$$

Finally, we rewrite 
$$
\sum_{\alpha \in \mathring{\Phi}^+ \backslash \mathring{\Phi}^+_J} \langle \alpha, \xi \rangle = \left \langle \left( \sum_{\alpha \in \mathring{\Phi}^+ \backslash \mathring{\Phi}^+_J} \alpha \right), \xi \right \rangle;
$$
we claim that this is zero. Indeed, since $\mathring{\Phi}^+ \backslash \mathring{\Phi}^+_J$ is a $W_J$-invariant set, the summation $\sum_{\alpha \in \mathring{\Phi}^+ \backslash \mathring{\Phi}^+_J} \alpha$ is $W_J$-fixed. This forces the pairing of this sum against any simple coroot $\alpha_i^\vee \in \mathring{Q}_J^\vee$ to be zero. Since $\xi \in \mathring{Q}_J^\vee$, this concludes the proof. 
\end{proof}

The next pair of lemmas gives relations between the orders $\leq_\eta$ and the coset-representative maps $\pi^{(\eta)}(\cdot)$. 

\begin{lemma} \label{cosetineqclaim1} For all $w \in W$, $\pi^{(\eta)}(w) \leq_\eta w$. 
\end{lemma}

\begin{proof}
Write by unique factorization $w=\pi^{(\eta)}(w)\pi_{(\eta)}(w)$. In $W(\eta)$, write 
$$
\pi_{(\eta)}(w)=s_{\gamma_1}\cdots s_{\gamma_k}
$$
as a reduced word in Coxeter generators $\{s_{\gamma_i}\}$ of $W(\eta)$. For $0 \leq j \leq k$, we set 
$$
w_j:=\pi^{(\eta)}(w)s_{\gamma_1} \cdots s_{\gamma_j}.
$$
By unique factorization, we have $\pi^{(\eta)}(w_j)=\pi^{(\eta)}(w)$ and $\pi_{(\eta)}(w_j)=s_{\gamma_1} \cdots s_{\gamma_j}$. Further, by Proposition \ref{lengthidentification}, we have 
$$
l_\eta(w_j)=l_\diamond (\pi^{(\eta)}(w))+j.
$$

Now, we rewrite $w_j = s_{\beta_1}\cdots s_{\beta_j} \pi^{(\eta)}(w)$, where $\beta_r:= \pi^{(\eta)}(w) \left(\gamma_r \right) \in \Phi^+_{re}$, as $\gamma_r \in \Phi^+_\eta$ and of course $\pi^{(\eta)}(w) \in W^{(\eta)}$. Then we can now write 
$$
w_{j+1}=s_\varepsilon w_j,
$$
where $\varepsilon:= s_{\beta_1}\cdots s_{\beta_j}\left(\beta_{j+1}\right) = \pi^{(\eta)}(w) s_{\gamma_1} \cdots s_{\gamma_j}(\gamma_{j+1})$; note that $\varepsilon \in \Phi^+_{re}$, since $\pi^{(\eta)}(w) \in W^{(\eta)}$ and by assumption $\pi_{(\eta)}(w)=s_{\gamma_1} \cdots s_{\gamma_k}$ was a reduced expression in Coxeter generators, which forces $s_{\gamma_1}\cdots s_{\gamma_j}(\gamma_{j+1}) \in \Phi^+_\eta$. 

Thus we have a sequence of covering relations in $\leq_\eta$ as follows:
$$
\pi^{(\eta)}(w)=w_0 <_\eta w_1 <_\eta w_2 <_\eta \cdots <_\eta w_k=w,
$$
proving the lemma. 
\end{proof}

\begin{lemma} \label{cosetineqclaim2}
Fix $u \in W^{(\eta)}$ a coset representative, and let $w \in W$ such that $u \leq_\eta w$. Then $u \leq_\eta \pi^{(\eta)}(w)$. 
\end{lemma}

\begin{proof}
We induct on $l_\eta(w)-l_\eta(u)$, the base case being trivial. 

If $\pi^{(\eta)}(w)=w$, then we are done. Else, as in the proof of Lemma \ref{cosetineqclaim1}, there exists some Coxeter generator $s_\gamma \in W(\eta)$ such that $w s_\gamma <_\eta w$ (where, previously, we denoted $ws_\gamma := w_{k-1}$). Also, as $u \in W^{(\eta)}$, note that $u <_\eta us_\gamma$. Finally, by an appropriate ``right-handed" variation of the diamond lemma  (or the \textit{Z-property}) we can conclude that 
$$
u \leq_\eta w, \ u \leq_\eta us_\gamma, \ ws_\gamma \leq_\eta w \implies u \leq_\eta ws_\gamma.
$$
Therefore by induction $u \leq_\eta \pi^{(\eta)}(ws_\gamma)=\pi^{(\eta)}(w)$.
\end{proof}

\begin{remark}
We do not supply details for this right-handed diamond lemma, proven by Dyer \cite{Dyer3}, which is applicable for a certain subset of (not necessarily simple) reflections. We refer to \cite{Dyer4}*{(1.19) and Example 1.20(iv)} to apply this result in our setting for $s_\gamma \in W(\eta)$ a Coxeter generator. 
\end{remark}

Taken together, Lemmas \ref{cosetineqclaim1} and \ref{cosetineqclaim2} give the following proposition. 

\begin{proposition} \label{ineqdescent}

The map $\pi^{(\eta)}$ is $\leq_\eta$-order preserving; that is, let $u, w \in W$ such that $u \leq_\eta w$. Then $\pi^{(\eta)}(u) \leq_\eta \pi^{(\eta)}(w)$.
\end{proposition}

\begin{proof}
Fix $u \in W$ and consider $w \in W$ with $u \leq_\eta w$. By Lemma \ref{cosetineqclaim1}, we know that 
$$
\pi^{(\eta)}(u) \leq_\eta u \leq_\eta w,
$$
so that $\pi^{(\eta)}(u) \leq_\eta w$. Now apply Lemma \ref{cosetineqclaim2} to get $\pi^{(\eta)}(u) \leq_\eta \pi^{(\eta)}(w)$. 
\end{proof}

Finally, we conclude this section with an alternative characterization of $\leq_\eta$, as given in \cite{Dyer4}*{Definition 1.11} (see also \cite{Dyer3}*{Proposition 1.2}) specialized to our case of cocharacter-twisted orders. We will make use of this characterization in the proof of Theorem \ref{face-intersection-1}.

\begin{proposition} \label{altorder}
Fix an appropriately dominant cocharacter $\eta$, and let $w \in W$ be arbitrary and $s_\gamma \in W$ a reflection for some positive real root $\gamma \in \Phi^+_{re}$. Then $\leq_\eta$ is the unique partial order on $W$ satisfying
$$
s_\gamma w \leq_\eta w \iff \gamma \in \Phi_w \oplus w(A'_\eta),
$$
where $\oplus$ denotes the symmetric set difference.
\end{proposition}

\section{Combinatorial structure on faces of $P_w^\lambda$} \label{hummingbirds}

In this section, we prove the primary combinatorial results which allow us to understand the faces of the Demazure weight polytope $P_\lambda^w$. To do this, we first introduce (in analogy with the orders $\leq_\diamond$) a \emph{twisted Demazure product} $\ast_\eta$ associated to the twisted orders $\leq_\eta$. Since the machinery developed in Section \ref{Demprodsection}, and in particular in Definition \ref{Demproddef} and Proposition \ref{Demprodmax}, only relied on our choice of order and the related length function satisfying the diamond lemma, the results therein are readily available for the twisted orders $\leq_\eta$ via Proposition \ref{twisteddiamond}. For example, we have the following proposition.

\begin{proposition} \label{twistedproductdef} 
Let $w, v \in W$. Then the set $\{ [e,w]v\}$ has a unique maximal element with respect to the twisted order $\leq_\eta$, and we can write 
$$
w \ast_\eta v:=\op{max}_{\leq_\eta}\{[e,w]v\}.
$$
\end{proposition}

\begin{remark} The existence of this maximum was also given in \cite{CD}*{Lemma 3.5}, again by exploiting the diamond lemma for $\leq_\eta$ (and more general twisted orders). However, the description and connection to Demazure products was not explored in loc. cit. 
\end{remark}

We can now state and prove the first major result of this section, which gives a ``first approximation" to the intersections of interest as in Task \ref{task-farce}. We adopt uniformly the notational conventions of Definition \ref{etagroupsnotation}.

\begin{theorem} \label{face-intersection-1}
Let $\eta$ be an appropriately dominant cocharacter, and consider the associated twisted order $\leq_\eta$ and twisted Demazure product $\ast_\eta$. Let $W(\eta)$ and $W^{(\eta)}$ be the associated reflection group and coset representatives, respectively, with maps $\pi_{(\eta)}: W \to W(\eta)$ and $\pi^{(\eta)}: W \to W^{(\eta)}$. 

Fix $w \in W$ arbitrary, and fix a coset representative $v \in W^{(\eta)}$. Then as intervals in the Coxeter group $W(\eta)$, we have 
$$
W(\eta) \cap \left(\pi^{(\eta)}(w \ast_\eta v) \right)^{-1} [e, w]v = [e, \pi_{(\eta)}(w \ast_\eta v)].
$$

\end{theorem}

\begin{proof}
Let $\Theta:=W(\eta) \cap \left(\pi^{(\eta)}(w \ast_\eta v) \right)^{-1} [e, w]v$, for convenience. We show the two containments $[e, \pi_{(\eta)}(w \ast_\eta v)] \subseteq \Theta$ and $\Theta \subseteq [e, \pi_{(\eta)}(w \ast_\eta v)]$ in turn. 

For the first containment, we induct on $l_{W(\eta)}(\pi_{(\eta)}(w \ast_\eta v)) - l_{W(\eta)}(y)$ for $y \in [e, \pi_{(\eta)}(w \ast_\eta v)]$. The base case corresponds to $y=\pi_{(\eta)}(w \ast_\eta v)$. This is of course an element of $W(\eta)$. By Proposition \ref{twistedproductdef}, write $w \ast_\eta v = qv$ for some unique $q \in [e,w]$. Then by unique factorization 
$$
\pi_{(\eta)}(w \ast_\eta v) = \left( \pi^{(\eta)}(w \ast_\eta v) \right)^{-1} qv,
$$
so in total we get $\pi_{(\eta)}(w \ast_\eta v) \in \Theta$. 

Now suppose that $s_\gamma y <_{W(\eta)} y$ in the $W(\eta)$ Bruhat order; note then that necessarily $\gamma \in \Phi^+_\eta$, where $\Phi^+_\eta$ is as in Definition \ref{etagroupsnotation}. Since $W(\eta) \subset W$ is a reflection subgroup, by \cite{Dyer2}*{Theorem 1.4} we in fact also get
$$
s_\gamma y < y
$$
in the inherited Bruhat order from $W$. This implies that $y^{-1}(\gamma) \in \Phi^-$; more specifically, as $\gamma \in \Phi^+_\eta$ and $y \in W(\eta)$, we have that $y^{-1}(\gamma) \in - \Phi^+_\eta$. By induction, write $y= \left( \pi^{(\eta)}(w \ast_\eta v) \right)^{-1} r v$ for some $r \in [e,w]$. This gives 
$$
\begin{aligned}
y^{-1}(\gamma) &= v^{-1} r^{-1} \pi^{(\eta)}(w \ast_\eta v) (\gamma) \\
\implies v y^{-1}(\gamma) &= r^{-1} \left( \pi^{(\eta)}(w \ast_\eta v) (\gamma) \right).
\end{aligned}
$$
Since $y^{-1}(\gamma) \in -\Phi^+_\eta$ and by assumption $v \in W^{(\eta)}$, we have $vy^{-1}(\gamma) \in \Phi^-$. On the other hand, since $\gamma \in \Phi^+_\eta$ and $\pi^{(\eta)}(w \ast_\eta v) \in W^{(\eta)}$, we have $\beta:=\pi^{(\eta)}(w \ast_\eta v) (\gamma) \in \Phi^+$. Then $r^{-1}(\beta) \in \Phi^-$ which implies that, in usual Bruhat order,
$$
s_\beta r < r;
$$
thus $s_\beta r \in [e,w]$. Finally, we write 
$$
\begin{aligned}
s_\beta r &= \pi^{(\eta)}(w \ast_\eta v) s_\gamma \left( \pi^{(\eta)}(w \ast_\eta v) \right)^{-1} r \\
&= \pi^{(\eta)}(w \ast_\eta v) s_\gamma y v^{-1} \\
\implies s_\gamma y &= \left(\pi^{(\eta)}(w \ast_\eta v) \right)^{-1} (s_\beta r) v \in \Theta,
\end{aligned}
$$
finishing the proof of $[e, \pi_{(\eta)}(w \ast_\eta v)] \subseteq \Theta$.

For the second containment, let $y \in \Theta$, and write $y:=\left(\pi^{(\eta)}(w \ast_\eta v)\right)^{-1} r v$ for some $r \in [e,w]$. Then by definition, 
$$
\pi^{(\eta)}(w \ast_\eta v) y = rv \leq_\eta w \ast_\eta v.
$$
Then fix $\gamma_1, \gamma_2, \dots, \gamma_k \in \Phi^+_{re}$ such that 
$$
\pi^{(\eta)}(w \ast_\eta v) y = s_{\gamma_k}\cdots s_{\gamma_1}(w \ast_\eta v) <_\eta s_{\gamma_{k-1}}\cdots s_{\gamma_1} (w \ast_\eta v) <_\eta \cdots <_\eta s_{\gamma_1} (w \ast_\eta v) <_\eta w \ast_\eta v.
$$
We rewrite this as 
$$
\pi^{(\eta)}(w \ast_\eta v) y=\pi^{(\eta)}(w \ast_\eta v) s_{\beta_k}\cdots s_{\beta_1} \pi_{(\eta)}(w \ast_\eta v) <_\eta \cdots <_\eta \pi^{(\eta)}(w \ast_\eta v) s_{\beta_1} \pi_{(\eta)}(w \ast_\eta v) <_\eta w \ast_\eta v,
$$
where $\beta_i:=\left(\pi^{(\eta)}(w \ast_\eta v)\right)^{-1} (\gamma_i)$; we claim that each $\beta_i$ is a positive root in $\Phi^+_\eta$. Indeed, by Proposition \ref{ineqdescent} since $\pi^{(\eta)}(w \ast_\eta v) y$ and $w \ast_\eta v$ have the same coset representative in $W^{(\eta)}$, every element in this chain also has coset representative $\pi^{(\eta)}(w \ast_\eta v)$. This forces each $s_{\beta_i} \in W(\eta)$ so that $\beta_i \in \Phi_\eta$, but since $\gamma_i \in \Phi^+_{re}$ by assumption and $\beta_i=\left(\pi^{(\eta)}(w \ast_\eta v)\right)^{-1}(\gamma_i)$, indeed we get $\beta_i \in \Phi^+_\eta$. 

For all $i$, by Proposition \ref{altorder} we have 
$$
s_{\gamma_{i+1}}\cdots s_{\gamma_1} (w \ast_\eta v) <_\eta s_{\gamma_i} \cdots s_{\gamma_1} (w \ast_\eta v) \iff \gamma_{i+1} \in \Phi_{s_{\gamma_i} \cdots s_{\gamma_1} (w \ast_\eta v)} \oplus s_{\gamma_i} \cdots s_{\gamma_1} (w \ast_\eta v) \left(A'_\eta \right),
$$
where again $\oplus$ is the symmetric difference. Suppose that $\gamma_{i+1} \in s_{\gamma_i} \cdots s_{\gamma_1} (w \ast_\eta v) \left(A'_\eta \right)$. Write, for some $\alpha \in A'_\eta$,

\begin{equation*}
\begin{gathered}
\gamma_{i+1} = s_{\gamma_i} \cdots s_{\gamma_1} (w \ast_\eta v) (\alpha) \\
\implies \left(\pi^{(\eta)}(w \ast_\eta v)\right)^{-1}s_{\gamma_1} \cdots s_{\gamma_i} (\gamma_{i+1}) = \pi_{(\eta)}(w \ast_\eta v) (\alpha) \\
\implies s_{\beta_1} \cdots s_{\beta_i} (\beta_{i+1}) = \pi_{(\eta)}(w \ast_\eta v) (\alpha)
\end{gathered}
\end{equation*}

But now, pairing both sides with $\eta$ produces a contradiction: since each $\beta_j \in \Phi^+_\eta$, we know that $s_{\beta_j} (\eta)=\eta$ and $\langle \beta_j, \eta \rangle =0$. So, the left hand side evaluates to zero. But the right hand side evaluates as 
$$
\langle \pi_{(\eta)}(w \ast_\eta v) (\alpha), \eta \rangle = \langle \alpha, \eta \rangle <0
$$
since by assumption $\alpha \in A'_\eta$. Therefore by the symmetric difference, we conclude 
$$
\gamma_{i+1} \in \Phi_{s_{\gamma_i} \cdots s_{\gamma_1} (w \ast_\eta v)}.
$$
Hence, we have for all $i$ that

\begin{equation*}
\begin{gathered}
(w \ast_\eta v)^{-1} s_{\gamma_1} \cdots s_{\gamma_i} (\gamma_{i+1}) \in \Phi^- \\
\implies \left(\pi_{(\eta)}(w \ast_\eta v)\right)^{-1} s_{\beta_1} \cdots s_{\beta_i} (\beta_{i+1}) \in \Phi^- \\
\implies s_{\beta_{i+1}} \cdots s_{\beta_1} \pi_{(\eta)} <_{W(\eta)} s_{\beta_i} \cdots s_{\beta_1} \pi_{(\eta)}(w \ast_\eta v) \text{ as a relation in $W(\eta)$.}
\end{gathered}
\end{equation*}

\noindent In particular, this gives 
$$
y = s_{\beta_k} \cdots s_{\beta_1} \pi_{(\eta)}(w \ast_\eta v) <_{W(\eta)} \cdots <_{W(\eta)} s_{\beta_1} \pi_{(\eta)}(w \ast_\eta v) <_{W(\eta)} \pi_{(\eta)}(w \ast_\eta v),
$$
so that $y \in [e, \pi_{(\eta)}(w \ast_\eta v)]$, as desired.
\end{proof}

\begin{remark} Purely combinatorially, Theorem \ref{face-intersection-1} gives a description of the elements in a translated interval with fixed coset representative as an interval in a reflection (or often, parabolic) subgroup. While our context is for a particular coset representative, Oh and Richmond \cite{OR} obtain similar results for non-translated intervals with arbitrary coset representatives. The methods therein also rely on (classical) Demazure product arguments; it would be interesting to see how to adapt the twisted Demazure product approach more generally. As a particular common corollary to Theorem \ref{face-intersection-1} and \cite{OR}, we recall the following earlier result of van den Hombergh \cite{vdH} and Billey, Fan, and Losonczy \cite{BFL}.
\end{remark}

\begin{corollary} \label{classicintersection}
Let $w \in W$, and fix a parabolic subgroup $W_J \subset W$. Then for any reduced expression $w=s_{i_1}\cdots s_{i_k}$, we have 
$$
W_J \cap [e,w] = [e, s_{J}],
$$
where we define $s_J := s_{i_{j_1}} \ast s_{i_{j_2}} \ast \cdots \ast s_{i_{j_r}}$, where $(s_{i_{j_1}}, \dots, s_{i_{j_r}})$ is the maximal subsequence of simple reflections in $W_P$ in the reduced expression for $w$. 
\end{corollary}

\begin{proof}
This follows from Theorem \ref{face-intersection-1} by setting $\eta$ to be an affine antidominant cocharacter with vanishing $J$ and $v:=e \in W^{(\eta)}$. Then it is easy to see from the definition of $\ast_\eta$ in this case that, for $w$ as above, $w \ast_\eta v = \pi_{(\eta)}(w \ast_\eta v) = s_J$.
\end{proof}

Finally, we apply these results to complete our Task \ref{task-farce} and thereby describe combinatorially the faces of $P_\lambda^w$. The final step is to note that, as stated, Theorem \ref{face-intersection-1} does not a priori apply to these faces, as the result is for coset representatives of $w \ast_\eta v$ in place of $w \ast_\diamond v$. However, the subsequent proposition takes care of this supposed difference.

\begin{proposition} \label{same-coset}
Let $\eta$ be an appropriately dominant cocharacter, and fix $w \in W$ and $v \in W^{(\eta)}$. Let $\ast_\eta$ be the twisted Demazure product associated to $\eta$, and let $\ast_\diamond$ be the Demazure product for the associated ``regular" order for $\eta$ (as in Proposition \ref{lengthidentification}). Then in $W^{(\eta)}$ we have 
$$
\pi^{(\eta)}(w \ast_\eta v) = \pi^{(\eta)}(w \ast_\diamond v).
$$
\end{proposition}

\begin{proof}
Write $w \ast_\eta v = qv$ and $w \ast_\diamond v = q'v$ for elements $q, q' \in [e,w]$. Since the Demazure products are the maximum elements in $[e,w]v$ for their respective orderings, we have that 
$$
\begin{aligned}
w \ast_\diamond v &\leq_\eta w \ast_\eta v, \\
w \ast_\eta v &\leq_\diamond w \ast_\diamond v.
\end{aligned}
$$

By Proposition \ref{ineqdescent}, we have $\pi^{(\eta)}(w \ast_\diamond v) \leq_\eta \pi^{(\eta)} (w \ast_\eta v)$, so that $l_\eta(\pi^{(\eta)}(w \ast_\diamond v)) \leq l_\eta(\pi^{(\eta)} (w \ast_\eta v))$. By Proposition \ref{lengthidentification}, this in fact gives 
$$
l_\diamond(\pi^{(\eta)}(w \ast_\diamond v)) \leq l_\diamond(\pi^{(\eta)} (w \ast_\eta v)).
$$

Likewise, we can consider the restriction of the ordering $\leq_\diamond$ to $W^{(\eta)}$; this restriction is order-preserving on coset representatives. (When $\leq_\diamond$ is the usual or opposite Bruhat order this is classical; see \cite{INS}*{Lemma 6.1.1} for the corresponding result for $\leq_\semi$.) In particular, we have 
$$
\pi^{(\eta)}(w \ast_\eta v) \leq_\diamond \pi^{(\eta)}(w \ast_\diamond v),
$$
so that $l_\diamond( \pi^{(\eta)}(w \ast_\eta v)) \leq l_\diamond(\pi^{(\eta)}(w \ast_\diamond v))$. Taken together, we get again by Proposition \ref{ineqdescent} that
$$
l_\eta ( \pi^{(\eta)}(w \ast_\eta v)) = l_\eta(\pi^{(\eta)}(w \ast_\diamond v)).
$$
But since $\pi^{(\eta)}(w \ast_\diamond v) \leq_\eta \pi^{(\eta)}(w \ast_\eta v)$, this forces
$$
\pi^{(\eta)}(w \ast_\eta v) = \pi^{(\eta)}(w \ast_\diamond v).
$$
\end{proof}

\begin{remark} While we will not make explicit use of it here, it is not hard to see the related fact that the restrictions of $\leq_\eta$ and $\leq_\diamond$ to $W^{(\eta)}$ coincide. 
\end{remark}

Combining Theorem \ref{face-intersection-1} and Proposition \ref{same-coset}, we arrive at the following result which finishes Task \ref{task-farce} and gives a description of faces of $P_\lambda^w$. 

\begin{theorem} \label{face-intersection-2}
Fix $\lambda$ a dominant integral weight and $w \in W$. Let $\eta$ be an appropriately dominant cocharacter and fix $v \in W^{(\eta)}$. Retaining previous notation, the collection of vertices $\{q \lambda | q \in [e,w]\}$ lying on the face $\mathcal{F}(v, \eta)$ of $P_\lambda^w$ corresponds to the interval  
$$
\Omega:=W(\eta) \cap \left(\pi^{(\eta)}(w^{-1} \ast_\diamond v) \right)^{-1} [e, w^{-1}]v = [e, \pi_{(\eta)}(w^{-1} \ast_\eta v)]
$$
in $W(\eta)$, via
$$
y \in \Omega \mapsto \underbrace{ v y^{-1} \left(\pi^{(\eta)}(w^{-1}*_\diamond v)\right)^{-1}}_{q} \lambda.
$$
\end{theorem}

\begin{remark} Recall that by Proposition \ref{convhull}, we have for $w \in W$ and dominant integral weight $\lambda$ that 
$$
P_\lambda^w = \op{conv}(\{v\lambda: v \leq w\}).
$$
With this description, the affine Demazure weight polytopes are certain generalizations of the \emph{Bruhat interval polytopes}, introduced by Tsukerman and Williams \cite{TW}, to the setting of affine Weyl groups for the intervals $[e,w]$ (see also the discussion in \cite{BJK}*{Remark 7.8}). In this perspective, Theorem \ref{face-intersection-2} gives another instance of the recurring theme that ``faces of Bruhat interval polytopes are again Bruhat interval polytopes." 
\end{remark}

\begin{remark} When $G$ is a finite-dimensional complex reductive group, in \cite{BJK}*{Proposition 7.9} we make the further connection that the faces $\mathcal{F}(v, \eta)$ are themselves the Demazure weight polytopes corresponding to certain Levi subalgebras. This realization was used, crucially, to give an inductive approach to questions on saturation of Demazure characters. In the present setting, we can similarly realize the faces $\mathcal{F}(v, \eta)$ for affine Demazure weight polytopes as weight polytopes themselves, corresponding to (not necessarily Levi) subalgebras determined by $\Phi^+_\eta$. We do not give details here, for brevity.
\end{remark}

\end{document}